\documentclass[letterpaper,10pt]{article}

\usepackage[square,numbers]{natbib}

\usepackage[utf8]{inputenc}   
\usepackage[T1]{fontenc}  

\usepackage{pgf,tikz}
\usetikzlibrary{arrows}
\usepackage{multicol}
\usepackage{amsthm}
\usepackage{amsmath}
\usepackage{amssymb}
\usepackage{amsfonts}
\usepackage{stmaryrd}
\usepackage{mathabx}
\usepackage{latexsym}
\usepackage{color}
\usepackage{pdflscape}
\usepackage{graphics,graphicx,graphpap}
\usepackage{multirow}
\usepackage{rotating}
\usepackage[new]{old-arrows}
\usepackage[all]{xy}
\usepackage[letterpaper]{geometry}
\usepackage{subfigure}
\usepackage{hyperref}
\usepackage{orcidlink}
\usepackage{authblk}

\theoremstyle{plain}

\usepackage{tocbibind}

\DeclareMathAlphabet{\mathpzc}{OT1}{pzc}{m}{it}
\newtheorem{theorem}{Theorem}
\newtheorem{prop}[theorem]{Proposition}
\newtheorem{cor}[theorem]{Corollary}
\newtheorem{lem}[theorem]{Lemma}
\newtheorem{conj}[theorem]{Conjecture}
\newtheorem{que}[theorem]{Question}
\theoremstyle{definition}
\newtheorem{exa}[theorem]{Example}

\newcommand{\hocolim}{\operatornamewithlimits{\mathrm{hocolim}}}
\newcommand{\colim}{\operatornamewithlimits{\mathrm{colim}}}
\newcommand*{\email}[1]{%
    \normalsize\href{mailto:#1}{#1}\par
    }

\title{Total cut complexes and their duals}
\author{Andr\'es Carnero Bravo\;\orcidlink{0009-0001-0221-5908}}
\affil{Centro de Ciencias Matemáticas, UNAM,\\ A.P. 61-3, Xangari, Morelia, Michoacán 58089, México\\ \email{carnero@matmor.unam.mx}}
\date{}
\begin{document}
\maketitle
\begin{abstract}
We study the total $d$-cut complexes and their Alexander duals. We give some 
results about the connectivity in general and in terms of the grith of the graph. For $d\geq3$, the homotopy type of these complexes is calculated for: $p$th power of a cycle with at least 
$(2r)d$ vertices where $p\leq r$; the $r$th power of a cycle with at least $2rd-(r-1)$ vertices where $r\geq3$; and the $2$th power of  a cycle with at least $3d$ vertices. 
These calculations solve a conjecture of 
Bayer, Denker, Milutinović, Rowlands, Sundaram and Xue. The homotopy type of the $2$-total cut complex for any $r$th power of a cycle with $r\geq3$ also is calculated, 
solving a conjecture of Chauhan, Shukla and Vinayak. We also study the complexes of cartesian products of paths and of cartesian products of complete 
graphs for the total $2$-cut complex.
\end{abstract}
\tableofcontents
\textbf{Acknowledgments.} The author wishes to thank Sheila Sundaram for pointing out that the homotopy type of total cut complexes of complete multipartite graphs was already calculated in \citep{cutcomplexesi}. This work was supported by UNAM Posdoctoral Program (POSDOC) and by the project PAPIIT IN101423.
\section{Introduction}
Since Lovasz proof of Kneser's Conjecture \citep{lovasz}, graph complexes have become one of the main objects of study in topological 
combinatorics (see for example \citep{delongueville,haxelltopcinind,jonsson,kozlovcombalgtop}). 
Here we study the homotopy type of two filtrations of graph complexes. These filtrations are related through the Alexander duality. 
The complexes of the filtrtions are the ones define in \citep{kimboundind}--which we will call bounded independence complexes-- and the total cut complexes defined in 
\citep{totcutcompl}.
The first family of complexes was defined to study the problem of finding colorful independent transversals for families of independent sets of a graph. These complexes have also been called robust clique complexes in \cite{filak} and in \citep{shella} 
similar complexes were define and were called clique-free complexes --these complexes are the bounded independence complexes of the complement graph.
Inspired by a result of Fröberg \citep{froberg} and the extension in \citep{eagonreiner}, the total $k$-cut complex of a graph was defined and lately these complexes have receive a lot of attention 
(see for example \citep{chandrakartoptotcutgrid,chauhantot2comcycl,shenhmtpcyclepwr}).
The relation between this two families has been studied in \citep{filak,froberg24,zaare2026vertex}.

Here we solve the following conjectures.
\begin{conj}\citep{totcutcompl}
Let $n\geq3d$ and $d\geq3$, then 
$$\Delta_d^t(C_n^2)\simeq\left\lbrace\begin{array}{cc}
\mathbb{S}^{2i+1}     & \mbox{ if }n=3d+i\mbox{ and } 0\leq i\leq d-2 \\
\mathbb{S}^{2d+i} & \mbox{ if }n=4d+i\mbox{ and } i\geq-1
\end{array}\right.$$
\end{conj}

\begin{conj}\citep{shenhmtpcyclepwr}
For $d\geq3$ and $n\geq(2d-1)r+1$
$$\Delta_d^t(C_n^r)\simeq\mathbb{S}^{n-2d}$$
\end{conj}
We prove these conjectures in Theorem \ref{thmtotcyclpwrs}. 
%More precisely, we show that $\Delta_d^t(C_n^r)\simeq\mathbb{S}^{n-2d}$ for $r\geq2$ and $n\geq2rd$ for all $d\geq2$. 
For $d=2$, in \citep{shenhmtpcyclepwr} was conjectured that $\Delta_2^t(C_n^r)\simeq\mathbb{S}^{n-4}$ for $n\geq3r+1$ and $r\geq3$. This conjecture was solve in \citep{chauhantot2comcycl} and there it also was 
shown that $\Delta_2^t(C_{2r+2}^r)\simeq\mathbb{S}^{r-1}$. In \citep{chauhantot2comcycl} a conjecture was made about the remaining cases. 
We calculate the homotopy type of $\Delta_2^t(C_n^r)$ for all $r\geq3$ and $n\geq2r+2$, solving 
Conjecture 4.1 from \citep{chauhantot2comcycl}. For all this we first determine the homotopy type of the bounded independence complex of powers of cycles, and for this we need the homotopy type of the bounded 
independence complex of connected chordal graphs and of the disjoint union of chordal graphs.

While it is known that the homotopy type of a simplicial complex $K$ does not determine the homotopy typee of its Alexander Dual $K^*$ (see for example \citep{minianalxdual}), 
if both complexes are simply connected and one of 
the complexes has the homotopy type of a wedge of spheres of the same dimension, then the other has the homotopy type of a wedge of spheres of the corresponding dimension. This follows from Alexander Duality 
(see Theorem \ref{dualidadalexander}) and a folklore result in algebraic topology (see Theorem \ref{homttyphomlgrp}). From this fact, our approach is to use the homotopy type of the complexes of one of the 
filtrations to determine the homotopy type of the complexes of the other filtration. With this approach in mind, we give some general results in the connectivity of the 
complexes of both filtrations, we think these results can be helpful in the study of both filtrations.

In the first section the necessary preliminaries are given. In the second section we define the filtrations and give some general tools for studying the complexes, in this section we show 
that chordal graphs have contractible bounded independence complexes. In third section is focused on the connectivity 
of the complexes, in particular we show that the girth of a graph can be used to give a lower bound in the connectivity of the total cut complex. In the last section we calculate the homotopy 
type of the complexes for cycle powers, for complete multipartite graphs and some disconnected graphs. In particular we show that the bounded independence complex of disjoint graphs has the homotopy type of a 
homotopy colimit over a particular poset, using this we give an example of $d$-bounded independence and total $d$-cut complexes with torsion for $d\geq3$.
For the total $2$-cut complex the homotopy type for cartesian products of paths and for cartesian 
products of complete graphs is determined.

\section{Preliminaries}
For a positive integer $k$ we take the set $\underline{k}=\{1,\dots,k\}$. 
For a finite set $X$ we take $\binom{X}{r}$ as the set of all the subsets of $X$ with cardinality $r$ and 
$\mathcal{P}(X)$ will denote the power set of $X$. 

\subsection{Graph theory}
We only consider simple graphs \textit{i.e.} no loops or multiedges are allowed. We will denote the set of vertices by $V(G)$ and $|V(G)|$ will denote the \textit{order} of the graph. The edge set will be denoted by 
$E(G)$. For a vertex $v$, its \textit{open neighborhood} is the set $N_G(v)=\{u\in V(G):\;uv\in E(G)\}$ and its \textit{closed neighborhood} is the set $N_G[v]=N_G(v)\cup\{v\}$. We say two vertices $u,v$ are adjacent 
if $uv$ is an edge.
Given $S$ a subset of $V(G)$ the \textit{induced subgraph} $G[S]$ is the graph with vertex set $S$ and edge set
$$E(G[S])=\{uv:\;u,v\in S\mbox{ and }uv\in E(G)\}.$$
For a vertex set $S$, we define $G-S$ as $G[V(G)-S]$. If $S=\{u\}$, the notation $G-u$ is used instead of $G-\{u\}$.

A vertex set $S\subseteq V(G)$ is \textit{independent} if no two vertices in $S$ are adjacent. The \textit{independence number} is the maximal cardinality possible for a independent set and it will be denoted by $\alpha(G)$.

For a positive integer $n$, we define the \textit{path} on $n$ vertices as the graph $P_n$ with vertex set $\underline{n}$ and edge set 
$E(P_n)=\{ij:\;|i-j|=1\}$. We define the \textit{cycle} on $n\geq3$ vertices as the graph $C_n$ with vertex set $\underline{n}$ and edge set 
$E(C_n)=\{ij:\;|i-j|=1\}\cup\{1n\}$. Lastly we define the \textit{complete graph} on $n$ vertices as the graph $K_n$ with vertex set $\underline{n}$ and edge set $\{ij:\; i\neq j\}$.

A $uv$-\textit{path} of a graph $G$ between two vertices $u,v$ is a sequence of vertices $w_1w_2\cdots w_l$ such that any two are distinct, $w_iw_{i+1}$ is an edge for $1\leq i\leq l-1$, 
$w_1=u$ and $w_l=v$. The \textit{length} of a path $P$ is the number of vertices minus one and we will denoted by $l(G)$. The \textit{distance} between two vertices is 
$$d_G(u,v)=\min\{l(P):P\mbox{ is a }uv-\mbox{path}\}.$$
A \textit{cycle} of  $G$ is a sequence of vertices $w_1\cdots w_l$ such that $w_iw_{i+1}$ is an edge for $1\leq i\leq l-1$, $w_1w_l$ is an edge, $w_1=w_l$ and any other two vertices are distinct. 
The \textit{length of a cycle} is equal to the number of vertices in the cycle. 
The \textit{girth} of a graph $G$ is equal to the minimal possible length of a cycle of $G$. If the graph does not have cycles we say the girth is $\infty$. 
We will denote the girth of a graph by $g(G)$.

We say a graph $G$ is \textit{chordal} if $G$ does not have induced cycles of length equal or greater to $4$. A vertex $v$ is called a \textit{simplicial vertex} if $G[N_G(v)]$ is a complete graph.
\begin{theorem}\citep{repfntgrpstintRl}\label{teosimplvrtx}
If $G$ is chordal, then $G$ has a simplicial vertex.
\end{theorem}

Given a graph $G$ and a positive integer $r$, the  $r$\textit{th power} of $G$ is the graph $G^r$ with vertex set $V(G)$ and edge set 
$$E(G^r)=\{uv:\;d_G(u,v)\leq r\}.$$
Notice that for the powers of paths and cycles we have that
$$\alpha(P_n^r)=\left\lceil\frac{n}{r+1}\right\rceil,\;\;\;\alpha(C_n^r)=\left\lfloor\frac{n}{r+1}\right\rfloor.$$

Given two graphs $G$ and $H$, the \textit{cartesian product} $G\oblong H$ is the graph with vertex set $V(G\oblong H)=V(G)\times V(H)$ 
and with edge set 
$$E(G\oblong H)=\{\{(u_1,v),(u_2,v)\}:\;\{u_1,u_2\}\in E(G)\}\cup\{\{(u,v_1),(u,v_2)\}:\;\{v_1,v_2\}\in E(H)\}.$$
For all the graph definitions not stated here we follow \citep{graphsanddigraphs}.

\subsection{Algebraic topology}
All the spaces are compactly generated. All the homology and cohomology groups are with integer coefficients. Now we wive some results and definitions that we will need. 

\begin{theorem}[Whitehead's theorem, {\citep[see Corollary 4.33]{hatcher}}]\label{whiteheadhomologia}
If $X$ and $Y$ are simply connected CW-complexes and there is a continuous map $f\colon X\longrightarrow Y$ such 
that $f_*:H_n(X)\longrightarrow H_n(Y)$ is an isomorphism for each $n$, then $f$ is an homotopy equivalence.
\end{theorem}

Notice that Whitehead's Theorem tell us that a simply connected CW-complex with trivial homology must be contractible. 
We will use the following folklore result (see \citep{homtphmlgrps} for a proof).
\begin{theorem}\label{homttyphomlgrp}
Let $X$ be a simply connected CW-complex such that the only non-zero reduced homology group is $\tilde{H}_d(X)\cong\mathbb{Z}^a$. Then 
$$X\simeq\bigvee_a\mathbb{S}^d$$
\end{theorem}
A more general result in the same vein is the following.
\begin{theorem}\citep[Example 4C.2]{hatcher}\label{gradconse}
If $X$ is an CW-complex simply connected such that $\tilde{H}_q(X)\cong\mathbb{Z}^a$, $\tilde{H}_{q+1}(X)\cong\mathbb{Z}^b$ for some 
$q\geq2$ and the rest of the homology groups are trivial, then 
$$X\simeq\displaystyle\bigvee_a\mathbb{S}^q\vee\bigvee_b\mathbb{S}^{q+1}.$$
\end{theorem}

Given three spaces $X,Y,Z$, maps $f:X\longrightarrow Y$ and $g:X\longrightarrow Z$, the \textit{homotopy pushout} or \textit{double mapping cylinder} is the space 
$$\hocolim\left(Y \longleftarrow X\longrightarrow Z\right)=Y\sqcup X\times[0,1]\times Z/\sim$$
where $(x,0)\sim f(x)$ and $(x,1)\sim g(x)$. We will use the following folklore result.

\begin{lem}\label{homocolimpegado}
Let $X,Y,Z$ be spaces with maps $f:Z\longrightarrow X$ and $g:Z\longrightarrow Y$ such that both maps are null-homotopic. Then
$$\hocolim\left(\mathcal{S}\right)\simeq X\vee Y\vee \Sigma Z$$
where 
\begin{equation*}
    \xymatrix{
    \mathcal{S}: & Y \ar@{<-}[r]^{g} & Z \ar@{->}[r]^{f} & X
    }
\end{equation*}
\end{lem}

For a couple of result we will need more general homotopy colimits but we will not define them (see \citep{dugger2008primer,cubicalhomotopy,zieglerhocolim} for the definition).
Given a poset $\mathcal{P}$, it can be taken as a category and for a functor $\mathcal{X}:\mathcal{P}\longrightarrow\mathrm{Top}$ we get the homotopy colimit
$\hocolim_\mathcal{P}\mathcal{X}$.
We are interested in diagrams where every map $\mathcal{X}(a<b)$ is an inclusion between CW-complexes and the poset $\mathcal{P}$ is finite, so $\mathcal{P}$ is a Reedy category and the diagram is cofibrant, thus
$\colim_\mathcal{P}\mathcal{X}\simeq\hocolim_\mathcal{P}\mathcal{X}$ \citep[see section 8.4.2]{cubicalhomotopy}.

Given a finite set $V$, a \textit{simplicial complex} $K$ on the vertex set $V$ is a subset of $\mathcal{P}(V)$ such that is $\sigma$ is in $K$ and $\tau$ is a subset of $\sigma$, then $\tau$ is in $\sigma$. 
A simplex $\sigma$ that is not contained in any other simplex is called \textit{facet}. The \textit{dimension} of a simplex $\sigma$ is $|\sigma|-1$.
If $K=\emptyset$, we say it is a \textit{void complex}. Notice that the void complex is different to 
the complex $\{\emptyset\}$. We will not distinguish between a simplicial complex and its geometric realization. Given a simplicial complex $K$, its \textit{$d$-skeleton} is the simplicial complex 
$$\mathrm{sk}_dK=\{\sigma\in K:\;|\sigma|=d+1\}.$$
For a finite set $X$ we take the simplicial complex $\Delta^X=\mathcal{P}(X)$. Taking $[n]=\underline{n}\cup\{0\}$, we take $\Delta^n=\Delta^{[n]}$.
The \textit{boundary} of $\Delta^n$ is the complex $\partial\Delta^n=\Delta^n-\{[n]\}\cong\mathbb{S}^{n-1}$.
We will use the following known facts:
\begin{itemize}
    \item If $\mathrm{sk}_dK=\mathrm{sk}_d\Delta^{V(K)}$, then $K$ is $(d-1)$-connected.
    \item For any $d\leq n$, we have that 
    $$\mathrm{sk}_d\Delta^n\simeq\bigvee_{\binom{n}{d+1}}\mathbb{S}^d.$$
\end{itemize}
Given two complexes $K,L$ we say $L$ is a \textit{subcomplex} of $K$ if $V(L)\subseteq V(K)$ and $L\subseteq K$.
For a simplex $\sigma$ in $K$, we define the following three subcomplexes:
\begin{itemize}
    \item The \textit{link} of $\sigma$ is the simplicial complex
    $$\mathrm{lk}(\sigma)=\{\tau\in K:\;\sigma\cap\tau=\emptyset\mbox{ and }\tau\cup\sigma\in K\}.$$
    \item The \textit{star} of $\sigma$ is the simplicial complex
    $$\mathrm{st}(\sigma)=\{\tau\in K:\;\tau\cup\sigma\in K\}.$$
    \item The \textit{deletion} of $\sigma$ is the simplicial complex
    $$\mathrm{del}(\sigma)=\{\tau\in K:\;\sigma\nsubseteq\tau\}.$$
\end{itemize}
For a vertex $v$ of $K$ we will use the notations $\mathrm{lk}(v)$, $\mathrm{st}(v)$ and $\mathrm{del}(v)$. Notice that we allow \textit{phantom vertices}, these are vertices that are in the vertex set but 
the singletons are not in the complex. For a phantom vertex the link and star are void while the deletion of the vertex is the whole complex.

Given two vertex disjoint complexes $K$ and $J$, their \textit{join} is the simplicial complex 
$$K*J=\{\tau\cup\sigma:\;\tau\in K \mbox{ and } \sigma\in J\}.$$

Given a complex $K$ on $n$ vertices, its \textit{Alexander Dual} is the complex 
$$K^*=\{\sigma\subseteq V(K)\colon\;V(K)-\sigma\notin K\}.$$
We will use the Alexander duals to study both filtrations, for this we need the following theorem.

\begin{theorem}\label{dualidadalexander}(see \citep{bjorneralexander})
Let $K$ be a simplicial complex with $n$ vertices, then
$$\tilde{H}_i(K)\cong\tilde{H}^{n-i-3}(K^*)$$
\end{theorem}

A \textit{cover} of $K$ is a family of subcomplexes 
$\mathcal{U}=\{L_1,\dots,L_n\}$ such that $K=L_1\cup\cdots\cup L_n$. The \textit{nerve} of a cover $\mathcal{U}$ is the simplicial complex $\mathcal{N}(\mathcal{U})$ with vertex set $\mathcal{U}$ and a subset of 
subcomplexes is a complex if in the intersection there is a non-empty simplex, these intersections will be called \textit{non-empty intersections}.
\begin{theorem}[Nerve Theorem, {\citep[See Theorem 10.6]{bjornertopmeth}}]\label{nrvthm}
Let $K$ be a simplicial complex and take $\mathcal{U}=\{L_1,\dots,L_n\}$ a cover of $K$ such that:
\begin{enumerate}
    \item $L_i$ is contractible for all $i$.
    \item Any non-empty intersection $L_{i_1}\cap\cdots\cap L_{i_m}$ is contractible.
\end{enumerate}
Then $K\simeq\mathcal{N}(\mathcal{U})$
\end{theorem}

We will say a cover $\mathcal{U}$ of a complex $K$ is a \textit{basis-like cover} if the non-empty intersections also are in $\mathcal{U}$. 
Following the idea in the proof of Theorem 3.1 of \citep{clicomgrppowers}, we will 
use the following theorem that can be deduced from Theorem 6 of \citep{mcordsnghmgrpshmty}.
\begin{theorem}\label{theopweakeqv}
Let $K,J$ be simplicial complexes and $p:K\longrightarrow J$ a map. If $\mathcal{U}$ is a basis-like cover such that $p|_{p^{-1}(U)}:p^{-1}(U)\longrightarrow U$ is a weak equivalence for all $U$ in $\mathcal{U}$,
then $p$ is a weak equivalence.
\end{theorem}

Notice that if $\mathcal{U}$ is a cover of $J$ which achieves the hypothesis of the Nerve Theorem \ref{nrvthm}, we can obtain a basis-like cover $\mathcal{U'}$ 
by taking all the non-empty intersections of $\mathcal{U}$. If $p:K\longrightarrow J$ is a simplicial map that achieves the hypothesis of Theorem \ref{theopweakeqv}, then the cover $p^{-1}(\mathcal{U'})$ achieves the 
hypothesis of the Nerve theorem.

Next lemma will alow us to study the homotopy type of the filtration for the disjoint union of two graphs.
\begin{lem}\label{lemunionorder}
For $k\geq3$, let $X_1,\dots,X_k$ be simplicial complexes such that:
\begin{itemize}
    \item[(a)] $\emptyset\neq X_i\cap X_{i+1}\neq\{\emptyset\}$ for $i\leq k-1$. 
    \item[(b)] For $i<j$ and any $i<l_1<\cdots<l_m<j$ 
    $$X_i\cap X_j\cap\bigcap_{r=1}^mX_{l_r}=X_i\cap X_j.$$
    \item[(c)] For any $i<k$ the inclusions $X_i\cap X_{i+1}\longhookrightarrow X_{i}$ and $X_i\cap X_{i+1}\longhookrightarrow X_{i+1}$ are null-homotopic.
\end{itemize}
Then 
$$\bigcup_{i=1}^kX_i\simeq\bigvee_{i=1}^kX_i\vee\bigvee_{i=1}^{k-1}\Sigma\left(X_i\cap X_{i+1}\right).$$
\end{lem}
\begin{proof}
The proof is by induction over $k$.
For $k=3$, notice that 
$$\bigcup_{i=1}^3X_i=\colim\left(X_1\longhookleftarrow X_1\cap(X_2\cup X_3)\longhookrightarrow X_2\cup X_3\right).$$
Now, by hypothesis we have that $X_1\cap X_3=X_1\cap X_2\cap X_3\subseteq X_1\cap X_2$, therefore
$$\bigcup_{i=1}^3X_i=\colim\left(X_1\longhookleftarrow X_1\cap X_2\longhookrightarrow X_2\cup X_3\right).$$
Because the maps are cofibrations, we have that the pushout is homotopy equivalent to the homotopy pushout \citep[see Proposition 3.6.17]{cubicalhomotopy}. By Lemma \ref{homocolimpegado} the union has the homotopy 
type desired. 
Before proceeding to the general case, we need to show that under the hypothesis of point (a) 
\begin{equation}\label{equaunion}
    X_1\cap\left(\bigcup_{i=2}^kX_i\right)=X_1\cap X_2.
\end{equation}
For $k=3$ this is true, assume it is true for $k$ and we take $X_1,\dots,X_{k+1}$ simplicial complexes that achieve (a), where $k+1\geq4$. Now
$$X_1\cap\left(\bigcup_{i=2}^kX_i\right)=(X_1\cap X_2)\cup\left(X_1\cap\left(\bigcup_{i=3}^kX_i\right)\right)=(X_1\cap X_2)\cup(X_1\cap X_3)$$
and because $X_1\cap X_3=X_1\cap X_2\cap X_3\subseteq X_1\cap X_2$, we have that (\ref{equaunion}) is true. 

Now assume the result is true for $k$ and take $X_1,\dots,X_{k+1}$ simplicial complexes that achieve the hypotheses of the lemma. Then 
$$\bigcup_{i=1}^{k+1}X_i=\colim\left(X_1\longhookleftarrow X_1\cap\left(\bigcup_{i=2}^{k+1}X_i\right)\longhookrightarrow\bigcup_{i=2}^{k+1}X_i\right)=$$
$$\colim\left(X_1\longhookleftarrow X_1\cap X_2\longhookrightarrow\bigcup_{i=2}^{k+1}X_i\right).$$
In the last pushout both maps are cofibrations and null-homotopic, calculating the homotopy pushout by Lemma \ref{homocolimpegado} and inductive hypothesis we obtain the result.
\end{proof}
An easy way to obtain a family like the required in Lemma \ref{lemunionorder} is two take two filtrations of simplicial complexes $L_1\geq L_2\geq\cdots\geq L_k$ and 
$J_1\leq J_2\leq\cdots\leq J_k$ such that in both filtrations all the inclusions are null-homotopic, then the complexes $X_i=L_i*J_i$ achieve the hypotheses.

For a finite poset $\mathcal{P}$ its \textit{order complex} $\Delta(\mathcal{P})$ is the simplicial complex with the elements of the poset as vertices and a set $\{a_1,\dots,a_d\}$ is a simplex if 
$a_1<\cdots<a_d$.

For positive integers $k\leq d$, a \textit{composition} of $d$ with $k$ parts is an ordered $k$-tuple $(d_1,\dots,d_k)$ of positive integers such that $d_i\geq1$ for all $i$ and $d_1+\cdots+d_k=d$. 
The set of all compositions of $d$ with $k$ parts will be denoted $C(d,k)$, it is known that $|C(d,k)|=\binom{d-1}{k-1}$ \citep[see Theorem 3.3]{combcompwords}. For $d>k$, we define $\mathcal{C}_{d,k}$ as 
the poset on $C(d,k)\cup C(d-1,k)\cup\cdots\cup C(k+1,k)$ where $(n_1,\dots,n_k)\leq(m_1,\dots,m_k)$ if $n_i\leq m_i$ for all $i$. Notice that for any element of $\mathcal{C}_{d,k}$ its coordinates can not be 
greater to $d$. We will denote $\hat{\mathcal{C}}_{d,k}$ to the poset obtained by adding $(1,1,\dots,1)$.

\begin{prop}\label{propcmk}
For two positive integers $d,k$, taking $m=d+k-1$ ,we have that
$$\Delta\left(\mathcal{C}_{m,k}\right)\simeq\left\lbrace\begin{array}{cc}
  *   & \mbox{if } k\leq d-1\\
  \displaystyle\bigvee_{\binom{k-1}{d-1}}\mathbb{S}^{d-2}   & k\geq d
\end{array}\right.$$
\end{prop}
\begin{proof}
We will show a covering of $\Delta\left(\mathcal{C}_{m,k}\right)$ such that the nerve have the homotopy type desired. For this, 
first we take $\Vec{e}_i$ the ordered $k$-tuple with the $i$-coordinate equal to $2$ and the rest equal to $1$. Next we take 
$\Vec{v}_i=\Vec{e}_i+(\Vec{e}_1+\cdots+\Vec{e}_k)$. For all $i$, $\Vec{v}_i$ is a vertex of $\Delta\left(\mathcal{C}_{m,k}\right)$ and 
$$\Delta\left(\mathcal{C}_{m,k}\right)=\bigcup_{i=1}^k\mathrm{st}(\Vec{v}_i).$$
For $i\neq j$ we have that $\mathrm{st}(\Vec{v}_i)\cap\mathrm{st}(\Vec{v}_j)=\mathrm{st}(\Vec{v}_i+\Vec{e}_j)=\mathrm{st}(\Vec{v}_j+\Vec{e}_i)\simeq*$.
In general, for $2\leq n\leq d-1$ and $\{i_1,\dots,i_n\}$, we have that 
$$\bigcap_{j=1}^n\mathrm{st}(\Vec{v}_{i_1})=\mathrm{st}(\Vec{v}_{i_1}+\Vec{e}_{i_2}+\cdots+\Vec{e}_{i_n})\simeq*.$$
Thus $\Delta\left(\mathcal{C}_{m,k}\right)\simeq*$ for $k\leq d-1$ and for $k\geq d$ the intersection of any $d-1$ complexes is contractible.
For $n\geq d$ and  $\{i_1,\dots,i_n\}$, we have that 
$$\bigcap_{j=1}^n\mathrm{st}(\Vec{v}_{i_1})=\{\emptyset\}$$
because any vertex in the intersection should have $n$ coordinates greater or equal to $2$ and the sum of the coordinates will bigger than $m$, but this can not happen.
Therefore the nerve is isomorphic to $\mathrm{sk}_{d-2}\Delta^{k-1}$ which has the homotopy type desired.
\end{proof}

The geometric realization of the nerve of a small category $\mathcal{I}$ is homeomorphic to the homotopy colimit of the functor sending every object in $\mathcal{I}$ to $*$ 
\citep[see Example 4.1 \& Remark 4.2]{dugger2008primer}. For a poset $\mathcal{P}$ the geometric realization of its nerve as a category is homeomorphic to the order complex of the poset 
\citep[see Examples 4.2-4.4]{friedmansimsets}. 
Given a small category $\mathcal{I}$ and two diagrams $F,G:\mathcal{I}\longrightarrow\mathrm{Top}$, a natural transformation $\alpha:F\longrightarrow G$ such that
$\alpha(i):F(i)\longrightarrow G(i)$ is a weak equivalence for every object $i$ in $\mathcal{I}$, then the natural map $\hocolim_\mathcal{I}F\longrightarrow\hocolim_\mathcal{I}G$ is a weak equivalence  
\citep[see Theorem 8.3.7]{cubicalhomotopy}. From all this we obtain the following corollary that is needed in last section.
\begin{cor}\label{corhocolimcmk}
For two positive integers $d,k$. If $m=d+k-1$ and  $\mathcal{X}:\mathcal{C}_{m,k}\longrightarrow\mathrm{Top}$ is a diagram such that 
$\mathcal{X}(\Vec{v})\simeq*$ for all $\Vec{v}$ in $C_{m,k}$, then
$$\hocolim_{\mathcal{C}_{m,k}}\mathcal{X}\simeq\Delta(\mathcal{C}_{m,k}).$$
\end{cor}

\section{Definition and some properties}
For a graph $G$, its total $d$-cut complex is
$$\Delta_d^t(G)=\{\sigma\subseteq V(G):\;\alpha(G[\sigma^c])\geq d\}$$
and its $d$-bounded independence complex is 
$$BI_d(G)=\{\sigma\subset V(G):\;\alpha(G[\sigma])<d\}$$
From the definition is clear that $\Delta_d^t(G)=BI_d(G)^*$. For 
$d=2$, $BI_2(G)$ is known as the clique complex of $G$.
Clearly each family of these complexes form a filtration:
$$\emptyset\subseteq\{\emptyset\}\subseteq BI_2(G)\subseteq BI_3(G)\subseteq\cdots\subseteq BI_{\alpha(G)}(G)\subseteq\Delta^{V(G)}$$
$$\emptyset\subseteq\Delta_{\alpha(G)}^t(G)\subseteq\cdots\subseteq\Delta_{2}^t(G)\subseteq\partial\Delta^{V(G)}\subseteq\Delta^{V(G)}$$

From Alexander duality we have that for a graph of order $n$
$$\tilde{H}_q(BI_d(G))\cong\tilde{H}^{n-q-3}(\Delta_d^t(G)),\;\;\tilde{H}_q(\Delta_d^t(G))\cong\tilde{H}^{n-q-3}(BI_d(G))$$

A well known result in clique complexes is the following:
If $N_G[v]\subseteq N_G[u]$, then $BI_2(G)\simeq BI_2(G-v)$ \citep[Proposition 3.2]{prisner}. We will generalize this result to all bounded independence complexes in next lemma.
\begin{lem}\label{lemneghbhd}
For all $d\geq2$ and $G$ a graph of order $n\geq2$:
\begin{enumerate}
    \item If there are distinct vertices $u,v$ such that $N_G[v]\subseteq N_G[u]$, then $BI_d(G)\simeq BI_d(G-v)$.
    \item If there are distinct vertices $u,v$ such that $N_G(v)\subseteq N_G(u)$, then 
    $$BI_d(G)\simeq BI_d(G-v)\vee\Sigma\mathrm{lk}_{BI_d(G)}(v).$$
\end{enumerate}
\end{lem}
\begin{proof}
We know that
$$BI_d(G)\simeq\hocolim\left(BI_d(G-v)\longhookleftarrow\mathrm{lk}_{BI_d(G)}(v)\longrightarrow*\right)$$
For (1), we take $\sigma$ a simplex of $\mathrm{lk}_{BI_d(G)}(v)$, then $\alpha(G[\sigma])<d$ and $\alpha(G[\sigma\cup\{v\}])<d$. If 
$$\alpha(G[\sigma\cup\{u,v\}])=d,$$ then there is a independent set $S$ in $\sigma$ such that $|S|=d-1$ and $S\cup\{u\}$ is an independent set. Because $N_G[v]\subseteq N_G[u]$, we have that 
$S\cup\{v\}$ is an independent set, but this can not happen. Therefore $\sigma\cup\{u\}$ is in $\mathrm{lk}_{BI_d(G)}(v)$ and $\mathrm{lk}_{BI_d(G)}(v)\simeq*$. 

For (2), with a similar argument as before we can see that the inclusion in the homotopy pushout can be decomposed as
$$\mathrm{lk}_{BI_d(G)}(v)\longhookrightarrow\{\{u\}\}*\mathrm{lk}_{BI_d(G)}(v)\longhookrightarrow BI_d(G-v).$$
Therefore the inclusion is null-homotopic. By Lemma \ref{homocolimpegado} we obtain the result.
\end{proof}

For connected chordal graphs it is known that the clique complex is contractible \citep[Lemma 3.1]{dochtermanengstrom}, with Lemma \ref{lemneghbhd} we can extend this result to all bounded independence complexes
using the same idea of the proof in \citep{dochtermanengstrom}.
\begin{prop}\label{propchordlbid}
If $G$ is a connected chordal graph, then $BI_d(G)\simeq*$ for all $d\geq2$.
\end{prop}
\begin{proof}
Assume $G$ has order $n\geq2$.
By Theorem \ref{teosimplvrtx} and Lemma \ref{lemneghbhd} we can order the vertices of $G$ in such way that 
$$BI_d(G)\simeq BI_d(G-v_1)\simeq BI_d(G-v_1-v_2)\simeq\cdots\simeq BI_d(G-v_1-\cdots-v_{n-1})\simeq*.$$
\end{proof}
Last Proposition tell us that the bounded independence complex of trees and  powers of paths are contractible for $d\geq2$. As we will see, the bounded independence complex of a forest can be non-contractible.

Now we give some results about the links and the deletions of the total cut complex.
\begin{lem}\citep{totcutcompl}\label{lemmlnkdelta}
Let $G$ be a graph and fix $d\geq2$. If $\sigma$ is a simplex of $\Delta_d^t(G)$, then
$\mathrm{lk}(\sigma)=\Delta_d^t(G-\sigma)$.
\end{lem}

\begin{lem}\label{delcom}
Let $G$ be a graph such that $\alpha(G)\geq d\geq2$ and take a vertex $v$. 
\begin{itemize}
    \item[(a)] If $\alpha(G-v)\leq d-1$, then $\mathrm{del}(v)=\Delta_d^t(G)=\Delta^{N_G(v)}*\Delta_{d-1}^t(G-N_G[v])$.
    \item[(b)] If $\alpha(G-N_G[v])\leq d-2$, then $\mathrm{del}(v)=\Delta_d^t(G-v)$ and $\Delta_d^t(G)=\{\{v\}\}*\Delta_d^t(G-v)$.
    \item[(c)] If $\alpha(G-v)\geq d$ and $\alpha(G-N_G[v])\geq d-1$, then 
$$\mathrm{del}_{\Delta_d^t(G)}(v)=\Delta_d^t(G-v)\cup\left(\Delta^{N_G(v)}*\Delta_{d-1}^t(G-N_G[v])\right)$$
\end{itemize}
\end{lem}
\begin{proof}
\textit{(a)} If $\alpha(g-v)\leq d-1$, then all independent sets of cardinality $d$ must contain $v$. Therefore $v$ is a phantom vertex of the complex and 
$\mathrm{del}(v)=\Delta_d^t(G)$. Take $\sigma$ a facet of $\Delta_d^t(G)$, then $N_G(v)$ is contained in $\sigma$ and 
$\sigma'=\sigma-N_G(v)$ is a simplex of $\Delta_{d-1}(G-N_G[v])$. Thus $\mathrm{del}(v)$ is 
contained in $\Delta^{N_G(v)}*\Delta_{d-1}(G-N_G[v])$. The other inclusion is clear.

\textit{(b)} If  $\alpha(G-N_G[v])\leq d-2$, then there is no independent set of cardinality $d$ that contains $v$. 
Then for any simplex $\sigma$ in $\Delta_d^t(G-v)$, we have that $\sigma\cup\{u\}$ is a simplex of $\Delta_d^t(G)$. 
Also, for any $\tau$ facet of $\Delta_d^t(G)$ we have that $v$ is in $\tau$ and $V(G-v)-(\tau-\{v\})$ is and independent set of cardinality $d$.
Therefore $\Delta_d^t(G)=\{\{v\}\}*\Delta_d^t(G-v)$ and $\mathrm{del}(v)=\Delta_d^t(G-v)$.

\textit{(c)} Lastly assume that $\alpha(G-v)\geq d$ and $\alpha(G-N_G[v])\geq d-1$.
It is clear that 
$$\mathrm{del}_{\Delta_d^t(G)}(v)\supseteq\Delta_d^t(G-v)\cup\left(\Delta^{N_G(v)}*\Delta_{d-1}^t(G-N_G[v])\right).$$ 
For the other inclusion, let $\tau$ be a facet of $\mathrm{del}_{\Delta_d^t(G)}(v)$. 
If $\tau$ is a facet of $\Delta_d^t(G)$, then $V(G)-\tau$ is an independent set of 
size $d$ which contains $v$ and $(V(G)-\tau)\cap N_G(v)=\emptyset$. Thus $N_G(v)$ is contained in $\tau$.
We take $\tau'=\tau-N_G(v)$. 
Then $(V(G)-\{v\})-\tau'$ is an independent set of $G-N[v]$ and it has cardinality $d-1$. Therefore
$\tau$ is in $\Delta^{N_G(v)}*\Delta_{d-1}^t(G-N_G[v])$. 

Assume $\tau$ is not a facet of $\Delta_d^t(G)$, then 
$|V(G)-\tau|\geq d+1$ and there is an independent set $S$ of size $d$ in $V(G)-\tau$. Now, $v$ can not be in $S$. If $v$ is 
in $S$, then $\sigma=V(G)-S$ is a simplex of $\mathrm{del}_{\Delta_d^t(G)}(v)$ and $\tau$ is a subset of $\sigma$, but this can not happen. 
From this, $V(G)-S\cup\{v\}$ is a facet of $\mathrm{del}_{\Delta_d^t(G)}(v)$ and therefore $V(G)-S\cup\{v\}=\tau$. From all this, we have that 
$\tau$ is a simplex of $\Delta_d^t(G-v)$.
\end{proof}

\begin{prop}\label{homosusp}
Let $G$ be a graph such that %$\delta(G)\geq1$, 
$\alpha(G)\geq d\geq2$ and let $v$ be a vertex of $G$. If $\alpha(G-v)\geq d$ and  $\alpha(G-N_G[v])\geq d-1$, then
$$\Delta_d^t(G)\simeq\sum\left(\Delta_d^t(G-v)\cap\left(\Delta^{N_G(v)}*\Delta_{d-1}^t(G-N_G[v])\right)\right)$$
\end{prop}
\begin{proof}
By Lemma \ref{delcom} 
$$\mathrm{del}_{\Delta_d^t(G)}(v)=\Delta_d^t(G-v)\cup\left(\Delta^{N_G(v)}*\Delta_{d-1}^t(G-N_G[v])\right)$$
By Lemma \ref{lemmlnkdelta} $\mathrm{lk}_{\Delta_d^t(G)}(v)=\Delta_d^t(G-v)\neq\emptyset$, thus
$$\Delta_d^t(G)=\left(\{v\}*\Delta_d^t(G-v)\right)\cup\left(\Delta^{N_G(v)}*\Delta_{d-1}^t(G-N_G[v])\right)$$
Therefore $$\Delta_d^t(G)\simeq\hocolim\left( X\longhookleftarrow\Delta_d^t(G-v)\cap\left(\Delta^{N_G(v)}\ast\Delta_{d-1}^t(G-N_G[v])\right)\longhookrightarrow Y\right)$$
where $X=\left(\{v\}\ast\Delta_d^t(G-v)\right)$ and $Y=\left(\Delta^{N_G(v)}\ast\Delta_{d-1}^t(G-N_G[v])\right)$.
Because $X$ and $Y$ are contractible, by Lemma \ref{homocolimpegado} we obtain the result.    
\end{proof}

\section{Connectivity results}
In this section we give some results on the connectivity of the complexes of both filtrations. We begin with some result on the connectivity for pairs of bounden independence complexes 
and a more general result on the connectivity of the bounded independence complex. We show how the girth and the order give us bounds for the connectivity of the total cut complex and for 
the $2$-total cut complex we will see that these two parameters can determine the homotopy type.

\begin{prop}\label{connpair}
For any graph $G$ and $d\geq2$ the pair $(BI_{d+1}(G),BI_d(G))$ is $(d-2)$-connected.
\end{prop}
\begin{proof}
Notice that in $BI_{d+1}(G)-BI_d(G)$ the remaining simplices have dimension at least $d-1$, 
thus $(BI_{d+1}(G),BI_{d}(G))$ is $(d-2)$-connected (see \citep[Corollary 4.12]{hatcher}).
\end{proof}
For a graph $G$ and $r\geq1$, we have that $BI_d(G^r)\subseteq BI_d(G^{r+1})$, so is natural to ask if we can say something about the connectivity of the pair.
\begin{prop}
Let $G$ be a connected graph. Then $(BI_d(G^{r+1}),BI_d(G^r))$ is $(d-1)$-connected for $d\geq3$ and $r\geq1$.
\end{prop}
\begin{proof}
Notice that $(BI_d(G^{r+1}),BI_d(G^r))$ is $(d-2)$-connected because in $BI_d(G^{r+1})-BI_d(G^r)$ all the simplices have dimension at least $d-1$ (see \citep[Corollary 4.12]{hatcher}). 
Now, let $\sigma$ be a $(d-1)$-dimensional simplex in 
$BI_d(G^{r+1})-BI_d(G^r)$. Then $\sigma$ is an independent set in $G^r$ but not in $G^{r+1}$, thus  
any two vertices in $\sigma$ are at distance at least $r+1$ in $G$ and there must be a pair at this exact distance. 
Let $u,v$ in $\sigma$ be such that $d_G(u,v)=r+1$ and take  $uw_1\cdots w_rv$ a $uv$-path of minimal length in $G$. Suppose that 
$\sigma=\{u,v,x_1,\dots,x_{d-2}\}$. Then $\tau=\{w_r,u,v,x_1,\dots,x_{d-2}\}$ is a simplex of $BI_d(G^{r+1})$. In the chain complex for the pair we have that
$$\partial\left([w_r,u,v,x_1,\dots,x_{d-2}]\right)=[u,v,x_1,\dots,x_{d-2}]$$
Thus $H_{d-1}(BI_d(G^{r+1}),BI_d(G^r))\cong0$. By Proposition \ref{propconnbi}, $BI_d(G^r)$ is simply connected and 
by the Relative Hurewicz Theorem \citep[see Theorem 4.32]{hatcher} we have that the pair $(BI_d(G^{r+1}),BI_d(G^r))$ is $(d-1)$-connected.
\end{proof}

\begin{prop}\label{propconnbi}
Let $G$ be a graph, then $\mathrm{conn}(BI_d(G))\geq d-3$. If $G$ is connected, then $BI_3(G)$ is  simply connected.
\end{prop}
\begin{proof}
From its definition we have that $\mathrm{sk}_{d-2}BI_d(G)=\mathrm{sk}_{d-2}\Delta^{V(G)}$. 
Therefore $\mathrm{conn}\left(BI_d(G)\right)\geq d-3$. 

If $G$ is connected and has at most $3$ vertices, then $BI_3(G)\cong\Delta^{V(G)}\simeq*$. The only connected graph of order $4$ with 
independence number $3$ is $K_{1,3}$ and 
$$BI_3(K_{1,3})\cong\Delta^3-\{[3],\underline{3}\}\simeq*.$$
Assume the result it is true for any connected graph of order $n$ and take $G$ a connected graph of order $n+1$. Now, $G$ has a vertex $v$ such 
that $G-v$ is connected, thus 
$BI_3(G-v)=\mathrm{del}_{BI_3(G)}(v)$ is $1$-connected. We know that $BI_3(G)=\mathrm{del}_{BI_3(G)}(v)\cup\mathrm{st}_{BI_3(G)}(v)$. Notice that 
$BI_2(G-v)$ is a subcomplex of $\mathrm{lk}_{BI_3(G)}(v)$ and that $BI_2(G-v)$ is connected because $G-v$ is connected. By the Seifert–Van-Kampen Theorem $BI_3(G)$ is simply connected.
\end{proof}

The bound in last Proposition can not be improve. We will see that if $G$ is the disjoint union of $d$ or more paths, then $BI_d(G)$ has the homotopy type of a wedge of $(d-2)$-dimensional spheres.

\begin{theorem}\label{theoconngirth}
Let $G$ be a graph of order $n\geq2d+k$ such that $g(G)\geq2d$, then $\Delta_d^t(G)$ is $(k-1)$-connected. 
\end{theorem}
\begin{proof}
We take $\sigma$ in $\binom{V(G)}{k+1}$. If $G[V(G)-\sigma]$ is a forest, then 
$$\alpha(G[V(G)-\sigma])\geq\frac{2d-1}{2}$$
and $\sigma$ is a simplex of $\Delta_d^t(G)$. 
If $G[V(G)-\sigma]$ has a cycle, then 
$$\alpha(G[V(G)-\sigma])\geq\left\lfloor\frac{g(G)}{2}\right\rfloor\geq d$$
and $\sigma$ is a simplex of $\Delta_d^t(G)$.
Therefore $\mathrm{sk}_k\Delta_d^t(G)\cong\mathrm{sk}_k\Delta^{V(G)}$ and $\Delta_d^t(G)$ is $(k-1)$-connected.
\end{proof}

\begin{cor}
Let $G$ be a graph of order $n\geq2d+k$ such that $g(G)\geq2d$, then $H_q(BI_d(G))\cong0$ for any $q>2d-3$.
\end{cor}
\begin{proof}
If $n=2d$, then $G\cong C_{2d}$ and $BI_d(C_{2d})\cong\mathbb{S}^{2d-3}$ because the only minimal non-simplexes are the sets $U=\{1,\dots,2d-1\}$ and $W=\{2,\dots,2d\}$, thus 
$$BI_d(C_{2d})=\partial\Delta^{U}*\partial\Delta^{W}.$$
Assume $n>2d$.
By Theorem \ref{theoconngirth} and Hurewicz Theorem the first possibly non-trivial reduced homology group of $\Delta_d^t(G)$ is $\tilde{H}_k(\Delta_d^t(G))$. 
Let $T_q$ be the torsion subgroup of $H_q(\Delta_d^t(G))$. Then 
$$H^q(\Delta_d^t(G))\cong(H_q(\Delta_d^t(G))/T_q)\bigoplus T_{q-1}$$
\citep[see Corollary 3.3]{hatcher} 
and the smallest $q$ for which $\Delta_d^t(G)$ may have non-trivial cohomoloy is $k$. By Alexander duality we have that $H_{q}(BI_d(G))\cong0$ for $q>2d-3$.
\end{proof}

\begin{cor}
Let $G$ be a graph of order $n\geq6$ such that $g(G)\geq4$, then $\Delta_2^t(G)$ is contractible or it has the homotopy type of a wedge of spheres with $n-3, n-4$ as the only possible dimensions.
\end{cor}
\begin{proof}
Because $G$ has girth at least $4$, then $\tilde{H}^q(BI_2(G))\cong0$  for any $q\neq0,1$. Therefore 
$\tilde{H}_q(\Delta_2^t(G))\cong0$ for any $q\neq n-3,n-4$ and the only possible non-trivial homology groups are torsion free finitely generated. By Theorem  \ref{theoconngirth}
$\Delta_2^t(G)$ is simply connected. We obtain the result by Theorem \ref{gradconse}.
\end{proof}

\section{Homotopy type calculations}
In this section we calculate the homotopy type of both families of complexes for various families of graphs. In the first subsection we study 
the complexes for some disconnected graphs, these calculations will be needed in the third subsection. In the second subsection we work on 
complete multipartite graphs for the bounded independence complexes. In the third subsection we study cycles and some of their powers. In the last subsection we focus on the clique
complex and the total $2$-cut complex for the cartesian product of paths and for the cartesian product of complete graphs.
\subsection{Disconnected graphs}

Let $G$ be a disconnected graph with connected components $G_1,\dots,G_k$. For $m=d+k-1$ and for each $\Vec{c}=(d_1,\dots,d_k)$ in $C(m,k)$, we take 
$$BI_{\Vec{c}}(G)=\bigast_{j=1}^k BI_{d_j}(G_{j}).$$

\begin{lem}\label{lemdjtunionbid}
Let $G$ be a disconnected graph and let $G_1, . . . , G_k$ be its connected components. Then
$$BI_d\left(G\right)=\bigcup_{\Vec{c}\in C(d+k-1,k)}BI_{\Vec{c}}(G).$$
\end{lem}
\begin{proof}
For each $\Vec{c}$ in $C(d+k-1,k)$, if $p_i$ is the projection of the $i$-coordinate, we have that $p_i(\Vec{c})\leq d$ for all $i$. 
If $J=\{i_1,\dots,i_n\}$ are the indices such that 
$p_{i_j}(\Vec{c})>1$, then 
$$\sum_{j=1}^np_{i_j}(\Vec{c})=d+k-1-(k-n)=d+n-1.$$
Now take $\sigma$ in $BI_{\Vec{c}}(\underline{k})$, notice that $\sigma$ only has vertices in $G_{i_j}$ for $i_j$ in $J$. If we take $\sigma_{i_j}=V(G_{i_j})\cap\sigma$ and 
$d_{i_j}=p_{i_j}(\Vec{c})$, we have that 
$$\alpha(G[\sigma])=\sum_{j=1}^n\alpha(G_i[\sigma_{i_j}])\leq\sum_{j=1}^n(d_{i_j}-1)=d-1$$
thus $\sigma$ is a simplex of $BI_d\left(G\right)$.

Now we take $\sigma$ a simplex of $BI_d\left(G\right)$. If we take $\sigma_i=V(G_i)\cap\sigma$ and $d_i=\alpha(G[\sigma_i])+1$, then 
$\sigma_i$ is a simplex of $BI_{d_i}(G_i)$ and we have that 
$$\sum_{i=1}^k\alpha(G_i[\sigma_i])\leq d-1.$$
Thus 
$$\sum_{j=1}^kd_i\leq d+k-1$$
and $\sigma$ is a simplex of $BI_{\Vec{c}}(\underline{k})$ for any $\Vec{c}$ in $C(d+k-1,k)$ such that $(d_1,\dots,d_k)\leq\Vec{c}$ in $\mathcal{C}_{d+k-1,k}$.
\end{proof}

\begin{theorem}\label{theodiscnntbid}
For $d\geq3$, let $G_1,\dots,G_k$ be graphs such that $BI_l(G_i)$ is contractible for any $i$ and all $2\leq l\leq d$. For $G=G_1+\cdots+G_k$
\begin{itemize}
    \item[(a)] If $k\leq d-1$, then 
    $BI_d(G)\simeq*.$
    \item[(b)] If $k\geq d$, then 
    $$BI_d(G)\simeq\bigvee_{\binom{k-1}{d-1}}\mathbb{S}^{d-2}.$$
\end{itemize}
\end{theorem}
\begin{proof}
First notice  that $G_i$ is connected for all $i$, otherwise $BI_2(G_i)$ would not be contractible. By Lemma \ref{lemdjtunionbid}, we have that
$$BI_d(G)=\bigcup_{\Vec{c}\in C(d+k-1,k)}BI_{\Vec{c}}(G).$$
We take $p_i:(\mathbb{Z}_{>0})^m\longrightarrow\mathbb{Z}$ the $i$-projection, where $m=d+k-1$. For distinct $\Vec{c}_1,\dots,\Vec{c}_n$, we define
$t_i=\min\{p_i(\Vec{c}_j)\}_{j=1}^n$ and $c(\Vec{c}_1,\dots,\Vec{c}_n)=(t_1,\dots,t_m)$. Then the intersection
$$BI_{\Vec{c}_1}(G)\cap\cdots\cap BI_{\Vec{c}_n}(G)$$
is a non-empty intersection if and only if $c(\Vec{c}_1,\dots,\Vec{c}_n)\neq(1,\dots,1)$. Thus any non-empty intersection is associated to a unique element of $\mathcal{C}_{m,k}$.
Now, the complexes and the non-empty intersections are contractible. Thus, 
If we define $\mathcal{X}:\mathcal{C}_{m,k}\longrightarrow\mathrm{Top}$ by $\mathcal{X}(\Vec{c})=BI_{\Vec{c}}(\underline{k})$, we have that 
$$BI_d(G)=\colim_{\mathcal{C}_{m,k}}\mathcal{X}\simeq\hocolim_{\mathcal{C}_{m,k}}\mathcal{X}$$
and by Corollary \ref{corhocolimcmk} the complex has the homotopy type desired.
\end{proof}

Last theorem tell us that if $G$ is the disjoint union of at most $d-1$ $p$th 
powers of paths, then $BI_d(G)\simeq*$ for any $p\geq1$. Also notice that in the proof of last theorem we obtain the following 
corollary.

\begin{cor}
Let $G$ be a disconnected graph and let $G_1, . . . , G_k$ be its connected components. Then
$$BI_d(G)\simeq\hocolim_{\mathcal{C}_{m,k}}\mathcal{X}$$
where $m=d+k-1$ and $\mathcal{X}:\mathcal{C}_{m,k}\longrightarrow\mathrm{Top}$ is given by $\mathcal{X}(\Vec{c})=BI_{\Vec{c}}(\underline{k})$.
\end{cor}

\begin{theorem}\label{theodisconntotcut}
Let $G$ be a disconnected graph and let $G_1,\dots,G_k$ be its connected components. If $k\geq d+3$, $\Delta_l^t(G_i)$ is void or contractible for all $2\leq l\leq d$  and $BI_2(G_i)$ is simply connected for 
all $1\leq i\leq k$, then
$$\Delta_d^t(G)\simeq\bigvee_{\binom{k-1}{d-1}}\mathbb{S}^{n-d-1}.$$
\end{theorem}
\begin{proof}
By Proposition \ref{propconnbi}, $BI_l(G_i)$ is simply connected for $3\leq l\leq d$ and all $1\leq i\leq k$. By hypothesis $BI_2(G_i)$ is simply connected. Now, if $\Delta_l^t(G_i)$ is void, then 
$BI_l(G_i)=\Delta^{V(G_i)}\simeq*$. If $\Delta_l^t(G_i)$ is contractible, then $BI_l(G_i)$ is a simply connected complex with trivial reduced homology groups. By Theorem \ref{whiteheadhomologia}, $BI_l(G_i)$
is contractible. By Theorem \ref{theodiscnntbid}
$$BI_d(G_i)\simeq\bigvee_{\binom{k-1}{d-1}}\mathbb{S}^{d-2}.$$
Thus, by Theorem \ref{dualidadalexander}, $\Delta_d^t(G)$ has the homology of a wedge of $(n-d-1)$-spheres. By Theorem \ref{homttyphomlgrp}, we only need to see that $\Delta_d^t(G)$ is simply connected. 
Let $x,y,z$ by distinct vertices of $G$, because there are $k\geq d+3$ connected components, there is a independent set of cardinality $d$ disjoint from $\{x,y,z\}$. Therefore $\{x,y,z\}$ is a simplex of 
$\Delta_d^t(G)$ and $\mathrm{sk}_2\Delta_d^t(G)=\mathrm{sk}_2\Delta^{V(G)}$.
\end{proof}

The condition on the number of connected components was impose to get that the total cut complex was simply connected. It is natural to ask the following question.
\begin{que}
Under the hypothesis of Theorem \ref{theodisconntotcut} but for $k=d,d+1,d+2$ connected components it is true that 
$$\Delta_d^t(G)\simeq\bigvee_{\binom{k-1}{d-1}}\mathbb{S}^{n-d-1}$$
for all $d\geq2$?
\end{que}

The connectivity condition imposed on the clique complex in Theorem \ref{theodisconntotcut} was needed to use Theorem \ref{theodiscnntbid}. So one can ask the next question.

\begin{que}
Is Theorem \ref{theodisconntotcut} true without the condition on the clique complex?
\end{que}

A natural question is the following: What can be say about disconnected graphs in general? Next proposition answers this question for a particular case. 

\begin{theorem}\label{thdjntunion2bid}
Let $G_1,G_2$ be graphs such that $\iota_i:BI_i(G_j)\longhookrightarrow BI_{i+1}(G_j)$ is null-homotopic for all $2\leq i\leq d-1$ and $j=1,2$. 
Then
$$BI_d(G_1\sqcup G_2)\simeq\bigvee_{i=1}^dBI_i(G_1)*BI_{d+1-i}(G_2)\vee\bigvee_{i=1}^{d-1}\Sigma BI_i(G_1)*BI_{d-i}(G_2).$$
\end{theorem}
\begin{proof}
Notice that 
$$BI_d(G_1\sqcup G_2)=\bigcup_{i=1}^{d}BI_i(G_1)*BI_{d+1-i}(G_2)$$
and this union achieve the hypotheses of Lemma \ref{lemunionorder}. Therefore 
$$BI_d(G_1\sqcup G_2)\simeq\bigvee_{i=1}^{d}BI_i(G_1)*BI_{d+1-i}(G_2)\vee\bigvee_{i=1}^{d-1}\Sigma BI_i(G_1)*BI_{d-i}(G_2).$$
\end{proof}

For a graph $G$ we define $G_{+_q}=G\sqcup qK_1$.
 
\begin{cor}\label{plusconstructionbid}
Let $G$ be a graph such that the inclusion $\iota_i:BI_i(G)\longhookrightarrow BI_{i+1}(G)$ is null-homotopic for all $2\leq i\leq d-1$. 
For $d\geq3$, we have that 
$$BI_d(G_{+_q})\simeq\bigvee_{\binom{q}{d-1}}\mathbb{S}^{d-2}\vee\bigvee_{i=2}^{d}\bigvee_{\binom{q}{d-i}}\Sigma^{d-i}BI_{i}(G).$$
\end{cor}
\begin{proof}
Notice that 
$$BI_i(qK_1)\cong\mathrm{sk}_{i-2}\Delta^{q-1}\simeq\bigvee_{\binom{q-1}{i-1}}\mathbb{S}^{i-2},$$
thus $BI_i(qK_1)\longhookrightarrow BI_{i+1}(qK_1)$ is null-homotopic for $2\leq i\leq d-1$. By Theorem \ref{thdjntunion2bid} we have that 
$$BI_d(G_{+_q})\simeq\bigvee_{i=1}^{d}BI_i(G)*BI_{d+1-i}(qK_1)\vee\bigvee_{i=1}^{d-1}\Sigma BI_i(G)*BI_{d-i}(qK_1).$$
$$\simeq\bigvee_{i=1}^{d}\left(BI_i(G)*\bigvee_{\binom{q-1}{d-i}}\mathbb{S}^{d-i-1}\right)\vee\bigvee_{i=1}^{d-1}\Sigma\left( BI_i(G)*\bigvee_{\binom{q-1}{d-i-1}}\mathbb{S}^{d-i-2}\right)$$
and from this equivalence the result can be easily deduced.
\end{proof}

Now we give an example of a disconnected graph such $G$ that $BI_d(G_{+_{d-2}})$ does not have the homotopy type 
of a wedge of spheres for $d\geq3$.

\begin{figure}
    \centering
    \begin{tikzpicture}[line cap=round,line join=round,>=triangle 45,x=0.5cm,y=0.5cm]
\clip(1.5,-1) rectangle (13.5,10);
\draw (6,0)-- (9,0);
\draw (9,0)-- (11.43,1.76);
\draw (11.43,1.76)-- (12.35,4.62);
\draw (12.35,4.62)-- (11.43,7.47);
\draw (11.43,7.47)-- (9,9.23);
\draw (9,9.23)-- (6,9.23);
\draw (6,9.23)-- (3.57,7.47);
\draw (3.57,7.47)-- (2.65,4.62);
\draw (2.65,4.62)-- (3.57,1.76);
\draw (3.57,1.76)-- (6,0);
\draw (6,2.5)-- (9,2.5);
\draw (9,2.5)-- (9.93,5.35);
\draw (9.93,5.35)-- (7.5,7.12);
\draw (7.5,7.12)-- (5.07,5.35);
\draw (5.07,5.35)-- (6,2.5);
\draw (7.5,7.12)-- (7.5,4.5);
\draw (7.5,4.5)-- (9.93,5.35);
\draw (9,2.5)-- (7.5,4.5);
\draw (7.5,4.5)-- (6,2.5);
\draw (7.5,4.5)-- (5.07,5.35);
\draw (9,2.5)-- (11.43,1.76);
\draw (9,2.5)-- (9,0);
\draw (9,2.5)-- (6,0);
\draw (11.43,1.76)-- (9.93,5.35);
\draw (12.35,4.62)-- (9.93,5.35);
\draw (9.93,5.35)-- (11.43,7.47);
\draw (11.43,7.47)-- (7.5,7.12);
\draw (9,9.23)-- (7.5,7.12);
\draw (7.5,7.12)-- (6,9.23);
\draw (6,9.23)-- (5.07,5.35);
\draw (5.07,5.35)-- (3.57,7.47);
\draw (5.07,5.35)-- (2.65,4.62);
\draw (6,2.5)-- (2.65,4.62);
\draw (6,2.5)-- (3.57,1.76);
\draw (6,0)-- (6,2.5);
\begin{scriptsize}
\fill [color=black] (6,0) circle (1.5pt);
\draw[color=black] (6,-0.32) node {$2$};
\fill [color=black] (9,0) circle (1.5pt);
\draw[color=black] (9,-0.32) node {$1$};
\fill [color=black] (11.43,1.76) circle (1.5pt);
\draw[color=black] (11.75,1.76) node {$5$};
\fill [color=black] (12.35,4.62) circle (1.5pt);
\draw[color=black] (12.65,4.62) node {$4$};
\fill [color=black] (11.43,7.47) circle (1.5pt);
\draw[color=black] (11.75,7.5) node {$3$};
\fill [color=black] (9,9.23) circle (1.5pt);
\draw[color=black] (9,9.55) node {$2$};
\fill [color=black] (6,9.23) circle (1.5pt);
\draw[color=black] (6,9.55) node {$1$};
\fill [color=black] (3.57,7.47) circle (1.5pt);
\draw[color=black] (3.23,7.47) node {$5$};
\fill [color=black] (2.65,4.62) circle (1.5pt);
\draw[color=black] (2.33,4.62) node {$4$};
\fill [color=black] (3.57,1.76) circle (1.5pt);
\draw[color=black] (3.25,1.76) node {$3$};
\fill [color=black] (7.5,4.5) circle (1.5pt);
\draw[color=black] (7.5,4.1) node {$11$};
\fill [color=black] (6,2.5) circle (1.5pt);
\draw[color=black] (6.2,2.18) node {$9$};
\fill [color=black] (9,2.5) circle (1.5pt);
\draw[color=black] (9.25,2.7) node {$8$};
\fill [color=black] (9.93,5.35) circle (1.5pt);
\draw[color=black] (9.93,5.67) node {$7$};
\fill [color=black] (7.5,7.12) circle (1.5pt);
\draw[color=black] (7.5,7.5) node {$6$};
\fill [color=black] (5.07,5.35) circle (1.5pt);
\draw[color=black] (4.8,5) node {$10$};
\end{scriptsize}
\end{tikzpicture}
\caption{Graph $G$ of order $11$ such that $BI_2(G)\cong\mathbb{RP}^2$}\label{rp2}
\end{figure}

\begin{exa}\label{exbid}
We take graph $G$ in Figure \ref{rp2}. Thus $BI_2(G)\cong\mathbb{RP}^2$ and the inclusion 
$$\iota:BI_2(G)\longhookrightarrow BI_3(G)$$ is null-homotopic as $BI_3(G)$ is simply connected by Proposition \ref{propconnbi}. 
For $d\geq 4$, we have that $BI_d(G)\simeq*$ because $\alpha(G)=3$. 
The minimal non-simplices of $BI_3(G)$ are 
\begin{equation}
    \{1,3,11\},\{1,4,11\},\{2,4,11\},\{2,5,11\},\{3,5,11\},\{1,7,9\},\{2,7,10\},\{3,8,10\},\{4,6.8\},\{5,6,9\}
\end{equation}
If we take $J$ the complex with maximal simplices the complements of the first $5$ non-simplexes of $(2)$ and 
$L$ the complex with maximal simplices the complements of the last $5$ non-simplexes of $(2)$, then $\Delta_3^t(G)=J\cup L$. 
Notice that $J\simeq*\simeq L$, thus $\Delta_3^t(G)\simeq\Sigma(J\cap L)$.
By calculations done with the computer we know that $\tilde{H}_4(\Delta_3^t(G))\cong\mathbb{Z}$ and that the other groups are trivial. 
Now, $J\cap L$ is connected as for any pair $i,j$ in $\underline{10}$, the edge is $\{i.j\}$ is in $J\cap L$. Therefore 
$\Delta_3^t(G)$ is simply connected and $\Delta_3^t(G)\simeq\mathbb{S}^4$. Now, by Alexander Duality we know that 
$\tilde{H}_4(BI_3(G))\cong\mathbb{Z}$. By Proposition \ref{propconnbi} $BI_3(G)$ is simply connected, therefore $BI_3(G)\simeq\mathbb{S}^4$.
By Corollary \ref{plusconstructionbid}, for all $d\geq2$, 
$G_{+_{d-2}}$ is a graph such that 
$$BI_d(G_{+_{d-2}})\simeq\Sigma^{d-2}\mathbb{RP}^2\vee\bigvee_{d-2}\mathbb{S}^{d+1}.$$ 
\end{exa}

It is natural to ask what can be say about the total cut complex of $G$ if we add isolated vertices. If we add only one extra vertex and $\Delta_{d+1}^t(G)\longhookrightarrow\Delta_{d}^t(G)$ is null-homotopic, then
$\Delta_d(G_{+_1})\simeq\Delta_d^t(G)\vee\Sigma\Delta_{d+1}^t(G)$ \citep[see Theorem 3.11]{totcutcompl}. The Corollary of next Theorem we will show what happens for the case $\Delta_d^t(G_{+_{q}})$ under similar hypotheses.

\begin{theorem}\label{thdjuniontotcut}
Let $G_1,G_2$ be graphs such that $\alpha(G_1)+\alpha(G_2)\geq d$ and  $\iota_i:\Delta_{i+1}^t(G_j)\longhookrightarrow\Delta_{i}^t(G_j)$ is 
null-homotopic for all $1\leq i\leq\alpha(G_j)-1$ and $j=1,2$.Then
$$\Delta_d^t(G_1\sqcup G_2)\simeq\bigvee_{i=l}^u\Delta_i^t(G_1)*\Delta_{d-i}^t(G_2)\bigvee_{i=l}^{u-1}\Sigma\Delta_{i+1}^t(G_1)*\Delta_{d-i}^t(G_2),$$
where $u=\min{\alpha(G_1),d}$ and $l=\max\{0,d-\alpha(G_2)\}$.
\end{theorem}
\begin{proof}
The proof is analogous to the proof of Theorem \ref{thdjntunion2bid} using the facts that
$$\Delta_d^t(G_1\sqcup G_2)=\bigcup_{i=l}^u\Delta_i^t(G_1)*\Delta_{d-i}^t(G_2),$$
and
$$\left(\Delta_i^t(G_1)*\Delta_{d-i}^t(G_2)\right)\cap\left(\Delta_{i+1}^t(G_1)*\Delta_{d-i-1}^t(G_2)\right)=\Delta_{i+1}^t(G_1)*\Delta_{d-i}^t(G_2).$$
Using Lemma \ref{lemunionorder} we obtain the result.
\end{proof}

\begin{cor}\label{plusconstructiontot}
Let $G$ be a graph such that $\alpha(G)=k\geq2$ and $\iota_i:\Delta_{i+1}^t(G)\longhookrightarrow\Delta_{i}^t(G)$ is null-homotopic for all $1\leq i\leq k-1$.
For $q\geq1$ such that $q+k\geq d$ we take $u=\min\{k,d\}$ and $l=\max\{0,d-q\}$, then
$$\Delta_d^t(G_{+_{q}})\simeq\bigvee_{i=l}^{u}\bigvee_{\binom{q}{d-i}}\Sigma^{q-d+i}\Delta_{i}^t(G).$$
%\begin{itemize}
%    \item[(a)] If $q,k\geq d$, then 
 %   $$\Delta_d^t(G_{+_{q}})\simeq\bigvee_{\binom{q}{d-1}}\mathbb{S}^{n+q-d-1}\bigvee_{i=2}^{d}\bigvee_{\binom{q}{d-i}}\Sigma^{q-d+i}\Delta_{i}^t(G).$$
  %  \item[(b)] If $q\geq d>k$, then 
   % $$\Delta_d^t(G_{+_{q}})\simeq\mathbb{S}^{n+q-d-1}\vee\bigvee_{i=2}^{k}\bigvee_{\binom{q}{d-i}}\Sigma^{q-d+i}\Delta_{i}^t(G).$$
    %\item[(c)] If $k\geq d>q$, then
    %$$\Delta_d^t(G_{+_{q}})\simeq\bigvee_{i=d-q}^{d}\bigvee_{\binom{q}{d-i}}\Sigma^{q-d+i}\Delta_{i}^t(G).$$
    %\item[(d)] If $q,k<d$, then
    %$$\Delta_d^t(G_{+_{q}})\simeq\bigvee_{i=d-q}^{k}\bigvee_{\binom{q}{d-i}}\Sigma^{q-d+i}\Delta_{i}^t(G).$$
%\end{itemize}
%Let $G$ be a graph such that $\alpha(G)\geq d\geq4$, $\Delta_d^t(G)$ is connected and the inclusion $\iota_i:\Delta_{i+1}^t(G)\longhookrightarrow\Delta_{i}^t(G)$ is null-homotopic for all $2\leq i\leq d-1$. 
%Then
%$$\Delta_d^t(G_{+_{d-2}})\simeq\bigvee_{i=2}^{d}\bigvee_{\binom{d-2}{d-i}}\Sigma^{i-2}\Delta_{i}^t(G).$$
\end{cor}
\begin{proof}
Notice that 
$$\Delta_i^t(qK_1)\cong\mathrm{sk}_{q-i-1}\Delta^{q-1}\simeq\bigvee_{\binom{q-1}{q-i}}\mathbb{S}^{q-i-1}.$$
Therefore $\iota_i:\Delta_{i+1}^t(qK_1)\longhookrightarrow\Delta_{i}^t(qK_1)$ is null-homotopic $1\leq i\leq q-2$. By Theorem 
\ref{thdjuniontotcut} we have that 
$$\Delta_d^t(G_{+_{q}})\simeq\bigvee_{i=l}^{u}\Delta_i^t(G)*\Delta_{d-i}^t(qK_1)\bigvee_{i=l}^{u-1}\Sigma\Delta_{i+1}^t(G)*\Delta_{d-i}^t(qK_1)$$
$$\simeq\bigvee_{i=l}^{u}\Delta_i^t(G)*\mathrm{sk}_{q-d+i-1}\Delta^{q-1}\;\vee\bigvee_{i=l}^{u-1}\Sigma\Delta_{i+1}^t(G)*\mathrm{sk}_{q-d+i-1}\Delta^{q-1}$$
$$\simeq\bigvee_{i=l}^{u}\left(\Delta_i^t(G)*\bigvee_{\binom{q-1}{q-d+i}}\mathbb{S}^{q-d+i-1}\right)\;\vee\bigvee_{i=l}^{u-1}\Sigma\left(\Delta_{i+1}^t(G)*\bigvee_{\binom{q-1}{q-d+i}}\mathbb{S}^{q-d+i-1}\right).$$
%Notice that $\Delta_0^t(G)*\mathrm{sk}_{q-d-1}\Delta^{V_+}\simeq*\simeq\Delta_d^t(G)*\mathrm{sk}_{q-1}\Delta^{V_+}$. 
Reindexing the second wedge and using the binomial identities we obtain the result.

%For the remaining cases we have that:
%$$(b)\;\displaystyle\Delta_d^t(G_{+_q})=\bigcup_{i=0}^{k}\Delta_i^t(G)*\mathrm{sk}_{q-d+i-1}\Delta^{V_+},$$
%$$(c)\;\displaystyle\Delta_d^t(G_{+_q})=\bigcup_{i=d-q}^{d}\Delta_i^t(G)*\mathrm{sk}_{q-d+i-1}\Delta^{V_+}\mbox{ and }$$
%$$(d)\;\displaystyle\Delta_d^t(G_{+_q})=\bigcup_{i=d-q}^{k}\Delta_i^t(G)*\mathrm{sk}_{q-d+i-1}\Delta^{V_+}.$$
\end{proof}

\begin{exa}
Taking the graph $G$ from Example \ref{exbid}, we know that $\Delta_3^t(G)\simeq\mathbb{S}^4$ and $\Delta_d^t(G)=\emptyset$ for $d\geq4$. By Alexander Duality
$$\tilde{H}_q(\Delta_d^t(G_{+_{d-2}}))\cong\left\lbrace\begin{array}{cc}
   \mathbb{Z}_2  & \mbox{for } q=6 \\
   \mathbb{Z}^{d-2}  & \mbox{for } q=5\\
   0 & \mbox{in other case}
\end{array}\right..$$
We will show that $\Delta_d^t(G_{+_{d-2}})$ has the homotopy type of a wedge of $5$-spheres and a copy of $\Sigma^5\mathbb{RP}^2$.
Notice that $\Delta_2^t(G)$ is simply connected as for any $3$-vertex set one of the $5$-cycles $1,2,3,4,5$ or $6,7,8,9,10$ shares at most one vertex with the set, thus 
$\mathrm{sk}_2\Delta_2^t(G)\cong\mathrm{sk}_2\Delta^{10}$. Then $\Delta_2^t(G)\simeq\Sigma^5\mathbb{RP}^2$ as it is the Moore space $M(\mathbb{Z}_2,6)$. Because $\Delta_2^t(G)$
is $5$-connected, we have that $\Delta_3^t(G)\longhookrightarrow\Delta_2^t(G)$ is null-homotopic. Now, $\Delta_1^t(G)\cong\mathbb{S}^9$, thus $\Delta_2^t(G)\longhookrightarrow\Delta_1^t(G)$ is null-homotopic.
From all this, by Corollary \ref{plusconstructiontot}, we get that 
$$\Delta_d^t(G_{+_{d-2}})\simeq\Sigma^5\mathbb{RP}^2\vee\bigvee_{d-2}\mathbb{S}^5.$$
\end{exa}

\subsection{Complete multipartite graphs}
In \citep[Proposition 4.6]{totcutcompl} the homotopy type of the total cut complex for complete bipartite graphs was calculated. For general multipartite graphs 
in \citep[see Theorems 7.8 $\&$ 7.9]{cutcomplexesi} the homotopy type is calculated for the cut complex. For multipartite graphs these complexes and the total cut complexes are the same. Here we do the 
calculations for the bounded independence complexes. Next proposition can be deduced using Proposition \ref{propconnbi} and the calculations of \citep{cutcomplexesi}, here we give a stray-forward demonstration.
\begin{prop}\label{propmultprtgrph}
For $k\geq2$, take $K_{n_1,\dots,n_k}$.
If $n_i\leq d-1$ for some $1\leq i\leq k$, then $BI_d(K_{n_1,\dots,n_k})\simeq*$. Otherwise 
$$BI_d(K_{n_1,\dots,n_k})\simeq\bigvee_{b(n_1,\dots,n_k)}\mathbb{S}^{k(d-1)-1}$$
where $\displaystyle b(n_1,\dots,n_k)=\prod_{i=1}^k\binom{n_i-1}{d-1}$
\end{prop}
\begin{proof}
Any simplex of $BI_d(K_{n_1,\dots,n_k})$ can not have more that $d-1$ vertices in each set of the vertex partition of $V(G)$. Thus
$$BI_d(K_{n_1,\dots,n_k})\cong\bigast_{i=1}^k\left(\mathrm{sk}_{d-2}\Delta^{n_i-1}\right)\simeq\bigast_{i=1}^k\left(\bigvee_{\binom{n_i-1}{d-1}}\mathbb{S}^{d-2}\right)$$
\end{proof}

\subsection{Cycles and their powers}
\begin{theorem}\citep{totcutcompl}\label{totcomcycles}
For $d\geq2$ and $n\geq2d$
$$\Delta_d^t(C_n)\simeq\mathbb{S}^{n-2d}$$
\end{theorem}
From last Theorem we obtain the following  corollary-- this result was obtained independently by a 
different method in \citep{zaare2026vertex}.
\begin{cor}\label{corhomtybidcn}
For $d\geq2$ and $n\geq2d$
$$BI_d(C_n)\simeq\mathbb{S}^{2d-3}$$
\end{cor}
\begin{proof}
For $d=2$ the result is trivial, as there are no triangles. 
By Proposition \ref{propconnbi} $BI_d(C_n)$ is simply connected for any $d\geq3$. By Theorem 
\ref{totcomcycles} and Theorem \ref{dualidadalexander} we have that $BI_d(C_n)$ has the homology of a $(2d-3)$-dimensional sphere. By 
Theorem \ref{homttyphomlgrp} we get that $BI_d(C_n)\simeq\mathbb{S}^{2d-3}$.
\end{proof}

\begin{figure}
\centering
\subfigure[Coloring $\psi_{3,12}$]{\begin{tikzpicture}[line cap=round,line join=round,>=triangle 45,x=0.5cm,y=0.5cm]
\clip(2,0.) rectangle (10,8.7);
\draw[color=blue] (5,1)-- (7,1);
\draw (7,1)-- (8.73,2);
\draw[color=orange] (8.73,2)-- (9.73,3.73);
\draw (9.73,3.73)-- (9.73,5.73);
\draw[color=green] (9.73,5.73)-- (8.73,7.46);
\draw (8.73,7.46)-- (7,8.46);
\draw[color=red] (7,8.46)-- (5,8.46);
\draw (5,8.46)-- (3.27,7.46);
\draw[color=yellow] (3.27,7.46)-- (2.27,5.73);
\draw (2.27,5.73)-- (2.27,3.73);
\draw[color=brown] (2.27,3.73)-- (3.27,2);
\draw (3.27,2)-- (5,1);
\draw (3.27,7.46)-- (2.27,3.73);
\draw (2.27,5.73)-- (3.27,2);
\draw (2.27,3.73)-- (5,1);
\draw (3.27,2)-- (7,1);
\draw (5,1)-- (8.73,2);
\draw (7,1)-- (9.73,3.73);
\draw (8.73,2)-- (9.73,5.73);
\draw (9.73,3.73)-- (8.73,7.46);
\draw (9.73,5.73)-- (7,8.46);
\draw (8.73,7.46)-- (5,8.46);
\draw (7,8.46)-- (3.27,7.46);
\draw (5,8.46)-- (2.27,5.73);
\begin{scriptsize}
\fill [color=blue] (5,1) circle (1.5pt);
\fill [color=blue] (7,1) circle (1.5pt);
\fill [color=orange] (8.73,2) circle (1.5pt);
\fill [color=orange] (9.73,3.73) circle (1.5pt);
\fill [color=green] (9.73,5.73) circle (1.5pt);
\fill [color=green] (8.73,7.46) circle (1.5pt);
\fill [color=red] (7,8.46) circle (1.5pt);
\fill [color=red] (5,8.46) circle (1.5pt);
\fill [color=yellow] (3.27,7.46) circle (1.5pt);
\fill [color=yellow] (2.27,5.73) circle (1.5pt);
\fill [color=brown] (2.27,3.73) circle (1.5pt);
\fill [color=brown] (3.27,2) circle (1.5pt);
\end{scriptsize}
\end{tikzpicture}}
\subfigure[Coloring $\psi_{3,15}$]{\begin{tikzpicture}[line cap=round,line join=round,>=triangle 45,x=0.8cm,y=0.8cm]
\clip(5,0.5) rectangle (10,6);
\draw[color=blue] (7,1)-- (8,1);
\draw[color=blue] (8,1)-- (8.91,1.41);
\draw (8.91,1.41)-- (9.58,2.15);
\draw[color=orange] (9.58,2.15)-- (9.89,3.1);
\draw[color=orange] (9.89,3.1)-- (9.79,4.1);
\draw (9.79,4.1)-- (9.29,4.96);
\draw[color=green] (9.29,4.96)-- (8.48,5.55);
\draw[color=green] (8.48,5.55)-- (7.5,5.76);
\draw (7.5,5.76)-- (6.52,5.55);
\draw[color=red] (6.52,5.55)-- (5.71,4.96);
\draw (5.71,4.96)-- (5.21,4.1);
\draw[color=yellow] (5.21,4.1)-- (5.11,3.1);
\draw (5.11,3.1)-- (5.42,2.15);
\draw[color=brown] (5.42,2.15)-- (6.09,1.41);
\draw (6.09,1.41)-- (7,1);
\draw (6.52,5.55)-- (5.21,4.1);
\draw (5.71,4.96)-- (5.11,3.1);
\draw (5.21,4.1)-- (5.42,2.15);
\draw (5.11,3.1)-- (6.09,1.41);
\draw (5.42,2.15)-- (7,1);
\draw (6.09,1.41)-- (8,1);
\draw[color=blue] (7,1)-- (8.91,1.41);
\draw (8,1)-- (9.58,2.15);
\draw (8.91,1.41)-- (9.89,3.1);
\draw[color=orange] (9.58,2.15)-- (9.79,4.1);
\draw (9.89,3.1)-- (9.29,4.96);
\draw (9.79,4.1)-- (8.48,5.55);
\draw[color=green] (9.29,4.96)-- (7.5,5.76);
\draw (8.48,5.55)-- (6.52,5.55);
\draw (7.5,5.76)-- (5.71,4.96);
\begin{scriptsize}
\fill [color=blue] (7,1) circle (1.5pt);
\fill [color=blue] (8,1) circle (1.5pt);
\fill [color=blue] (8.91,1.41) circle (1.5pt);
\fill [color=orange] (9.58,2.15) circle (1.5pt);
\fill [color=orange] (9.89,3.1) circle (1.5pt);
\fill [color=orange] (9.79,4.1) circle (1.5pt);
\fill [color=green] (9.29,4.96) circle (1.5pt);
\fill [color=green] (8.48,5.55) circle (1.5pt);
\fill [color=green] (7.5,5.76) circle (1.5pt);
\fill [color=red] (6.52,5.55) circle (1.5pt);
\fill [color=red] (5.71,4.96) circle (1.5pt);
\fill [color=yellow] (5.21,4.1) circle (1.5pt);
\fill [color=yellow] (5.11,3.1) circle (1.5pt);
\fill [color=brown] (5.42,2.15) circle (1.5pt);
\fill [color=brown] (6.09,1.41) circle (1.5pt);
\end{scriptsize}
\end{tikzpicture}}
\subfigure[Coloring  $\psi_{4,17}$]{\begin{tikzpicture}[line cap=round,line join=round,>=triangle 45,x=0.75cm,y=0.75cm]
\clip(1.5,0.5) rectangle (7.5,7);
\draw (5.93,1.36)-- (6.67,2.03);
\draw (7.12,2.93)-- (7.21,3.93);
\draw (6.94,4.89)-- (6.33,5.69);
\draw (5.48,6.21)-- (4.5,6.4);
\draw (3.52,6.21)-- (2.67,5.69);
\draw (2.06,4.89)-- (1.79,3.93);
\draw (1.88,2.93)-- (2.33,2.03);
\draw (3.07,1.36)-- (4,1);
\draw (4,1)-- (2.33,2.03);
\draw (3.07,1.36)-- (1.88,2.93);
\draw (2.33,2.03)-- (1.79,3.93);
\draw (5,1)-- (3.07,1.36);
\draw (1.88,2.93)-- (2.06,4.89);
\draw (1.79,3.93)-- (2.67,5.69);
\draw (2.06,4.89)-- (3.52,6.21);
\draw (2.67,5.69)-- (4.5,6.4);
\draw (3.52,6.21)-- (5.48,6.21);
\draw (4.5,6.4)-- (6.33,5.69);
\draw (5.48,6.21)-- (6.94,4.89);
\draw (6.33,5.69)-- (7.21,3.93);
\draw (6.94,4.89)-- (7.12,2.93);
\draw (7.21,3.93)-- (6.67,2.03);
\draw (7.12,2.93)-- (5.93,1.36);
\draw (6.67,2.03)-- (5,1);
\draw [color=blue] (4,1)-- (5.93,1.36);
\draw [color=blue] (5.93,1.36)-- (5,1);
\draw [color=blue] (5,1)-- (4,1);
\draw [color=orange] (6.67,2.03)-- (7.12,2.93);
\draw [color=green] (7.21,3.93)-- (6.94,4.89);
\draw [color=red] (6.33,5.69)-- (5.48,6.21);
\draw [color=yellow] (4.5,6.4)-- (3.52,6.21);
\draw [color=brown] (2.67,5.69)-- (2.06,4.89);
\draw [color=violet] (1.79,3.93)-- (1.88,2.93);
\draw [color=lime] (2.33,2.03)-- (3.07,1.36);
\begin{scriptsize}
\fill [color=blue] (4,1) circle (1.5pt);
\fill [color=blue] (5,1) circle (1.5pt);
\fill [color=blue] (5.93,1.36) circle (1.5pt);
\fill [color=orange] (6.67,2.03) circle (1.5pt);
\fill [color=orange] (7.12,2.93) circle (1.5pt);
\fill [color=green] (7.21,3.93) circle (1.5pt);
\fill [color=green] (6.94,4.89) circle (1.5pt);
\fill [color=red] (6.33,5.69) circle (1.5pt);
\fill [color=red] (5.48,6.21) circle (1.5pt);
\fill [color=yellow] (4.5,6.4) circle (1.5pt);
\fill [color=yellow] (3.52,6.21) circle (1.5pt);
\fill [color=brown] (2.67,5.69) circle (1.5pt);
\fill [color=brown] (2.06,4.89) circle (1.5pt);
\fill [color=violet] (1.79,3.93) circle (1.5pt);
\fill [color=violet] (1.88,2.93) circle (1.5pt);
\fill [color=lime] (2.33,2.03) circle (1.5pt);
\fill [color=lime] (3.07,1.36) circle (1.5pt);
\end{scriptsize}
\end{tikzpicture}}
\caption{Examples of $\psi_{d,n}$ in Theorem \ref{cyclepwerbid}}\label{colorinfpsi}
\end{figure}
Now we proceed to calculate the homotopy type of the bounded independence complex for cycle powers.
\begin{theorem}\label{cyclepwerbid}
For $d\geq3$, $n\geq(2r)d$ and $1\leq p\leq r$
$$BI_d(C_n^p)\simeq\mathbb{S}^{2d-3}$$
and each inclusion $BI_d(C_n)\longhookrightarrow BI_d(C_n^2)\longhookrightarrow\cdots\longhookrightarrow BI_d(C_n^r)$ is a homotopy equivalence.
\end{theorem}
\begin{proof}
The idea is to construct a map $\psi_{d,n}:\underline{n}\longrightarrow\underline{2d}$ such that for each $1\leq p\leq r$, this map will give us a simplcial map 
$BI_d(C_n^p)\longrightarrow BI_d(C_{2d})$ which is a homotopy equivalence. Assuming we have shown that the simplicial maps are homotopy equivalences, we get the following commutative diagram
\begin{equation*}
\xymatrix{
BI_d(C_n) \ar@{^(->}[rr]\ar@{->}_\simeq[rrrrd]&& BI_d(C_n^2)\ar@{^(->}[rr]\ar@{->}_\simeq[rrd]&&\cdots\ar@{^(->}[rr]&& BI_d(C_n^r)\ar@{->}^\simeq[lld]\\
&&&&BI(C_{2d})&&&&}
\end{equation*}
where all the downward arrows are homotopy equivalences, thus the inclusions also are homotopy equivalences.

Now, we define the function. We take $\psi_{d,n}:\underline{n}\longrightarrow\underline{2d}$
where:
$$\psi_{d,n}(i)=\left\lbrace\begin{array}{cc}
  t  & \mbox{ if } (l+1)(t-1)< i\leq(l+1)t \mbox{ and } 1\leq t\leq k\\
  &\\
  t+k & \mbox{ if } (l+1)k+l(t-1)<i\leq(l+1)k+lt \mbox{ and } 1\leq t\leq 2d-k
\end{array}\right.$$
where $n=l(2d)+k$ and $0\leq k\leq 2d-1$. The idea of the map is a coloring in which each chromatic class have at least $l\geq r$ consecutive vertices in the cycle (see Figure \ref{colorinfpsi}), 
the first $k$ classes will have $l+1$ vertices and the remaining will have $l$ vertices-- truly the only important thing is that each chromatic class have at least $r$ consecutive vertices.
Now, taking $V_0=\{2,\dots,2d\}$ and $V_1=\{1,\dots,2d-1\}$ we have that
$$BI_d(C_{2d})=\partial\Delta^{V_0}*\partial\Delta^{V_1}$$
and the maximal simplicies of $BI_d(C_{2d})$ are of the form 
$$\sigma_{_{i,j}}=\left(V_0-\{i\}\right)\cup\left(V_1-\{j\}\right)$$
with $i$ in $V_0$ and $j$ in $V_1$.
We take $1\leq p\leq r$  and $BI_d(C_n^p)$. Now we will see that $\psi_{d,n}$ define a simplicial map, for this 
we will show that for any $\sigma$ in $BI_d(C_n^p)$, $\left|\bigcup_{i\in\sigma}\psi_{d,n}(i)\cap V_j\right|<d$ for $j=0,1$. Assume this is not true and with out loss of generality assume 
$\left|\bigcup_{i\in\sigma}\psi_{d,n}(i)\cap V_0\right|=d$. We take a vertex $w_t$ in $\psi_{d,n}^{-1}(t)\cap\sigma$ 
for each $t$ in $V_0$, then $\{w_2,w_4,\dots,w_{2d}\}$ is an independent set of cardinality $d$ because $d_{C_n}(w_{t},w_{t+2})\geq r+1\geq p+1$ for all $t$ in $V_0$, but this can not happen. 
Therefore the image of any simplex of $BI_d(C_n^p)$ is a simplex and
$\psi_{d,n}$ define a simplicial map which we will denote the same. 
Now, we take $\mathcal{U}_{_{i,j}}=\Delta^{\sigma_{_{i,j}}}$ and for $S$ a non-empty subset of $V_0\times V_1$ we take
$$\mathcal{U}_{_{S}}=\bigcap_{(i,j)\in S}\mathcal{U}_{_{i,j}}.$$
We take $\mathcal{U}$ as the family of all the non-empty intersections $\mathcal{U}_{_{S}}$. 
Notice that $\psi_{d,n}^{-1}(\mathcal{U}_{_{S}})$ is isomorphic to $BI_d(H)$ where $H$ is either the $p$th power of a path or, because in any $\mathcal{U}_{_{S}}$ 
we have to erase at least one color class of $V_0$ and 
one of $V_1$, the disjoint union of at most $d-1$ $p$th powers of paths.
By Theorem \ref{theodiscnntbid}, we have that $\psi_{d,n}^{-1}(\mathcal{U}_{_{S}})$ is contractible, thus the restriction 
$$\psi_{d,n}|_{_{\psi_{d,n}^{-1}(\mathcal{U}_{_{S}})}}:\psi_{d,n}^{-1}(\mathcal{U}_{_{S}})\longrightarrow\mathcal{U}_{_{S}}$$
is a homotopy equivalence. Therefore $\psi_{d,n}$ is a homotopy equivalence by Theorem \ref{theopweakeqv}.
\end{proof}

A natural question is if the remaining cases also have the homotopy type of a sphere, next two theorems will show that the answer is yes for: all $n$ and $r=2$; for $(r+1)d\leq n\leq 2rd-1$ 
and $d\geq3$. For this we also will give a simplicial map $BI_d(C_n^2)\longrightarrow\partial\Delta^{p}$ for some $p$. Notice that in the proof of Theorem \ref{cyclepwerbid} we have shown that 
the family of subcomplexes of the subgraphs induced by all the colors but two give us a cover with nerve isomorphic to $BI_d(C_{2d}^p)$ and this cover achieves the hypothesis of the Nerve Theorem \ref{nrvthm}. 
In the following 
two theorems we will see that the colorings $BI_d(C_n^r)\longrightarrow\partial\Delta^{p}$ are homotopy equivalences by given a base-like cover for $\partial\Delta^p$, this again will give us a cover of $BI_d$ such 
that it is homotopy equivalent to the nerve of the covering-- in these cases the covering will by given by erasing one chromatic class at the time.

\begin{figure}
\centering
\subfigure[Coloring $\psi_{3,11}$]{\begin{tikzpicture}[line cap=round,line join=round,>=triangle 45,x=0.5cm,y=0.5cm]
\clip(0.9,0.5) rectangle (9.1,8.5);
\draw (6,1)-- (7.68,2.08);
\draw (8.51,3.9)-- (8.23,5.88);
\draw (6.92,7.39)-- (5,7.96);
\draw (3.08,7.39)-- (1.77,5.88);
\draw (1.77,5.88)-- (1.49,3.9);
\draw (1.49,3.9)-- (2.32,2.08);
\draw (2.32,2.08)-- (4,1);
\draw (2.32,2.08)-- (1.77,5.88);
\draw (2.32,2.08)-- (6,1);
\draw (4,1)-- (7.68,2.08);
\draw (1.49,3.9)-- (4,1);
\draw (6,1)-- (8.51,3.9);
\draw (7.68,2.08)-- (8.23,5.88);
\draw (8.51,3.9)-- (6.92,7.39);
\draw (8.23,5.88)-- (5,7.96);
\draw (6.92,7.39)-- (3.08,7.39);
\draw (5,7.96)-- (1.77,5.88);
\draw (3.08,7.39)-- (1.49,3.9);
\draw [color=blue] (4,1)-- (6,1);
\draw [color=orange] (7.68,2.08)-- (8.51,3.9);
\draw [color=green] (8.23,5.88)-- (6.92,7.39);
\draw [color=red] (5,7.96)-- (3.08,7.39);
\begin{scriptsize}
\fill [color=blue] (4,1) circle (1.5pt);
%\draw[color=blue] (4,0.7) node {$1$};
\fill [color=blue] (6,1) circle (1.5pt);
%\draw[color=blue] (6,0.7) node {$2$};
\fill [color=orange] (7.68,2.08) circle (1.5pt);
%\draw[color=orange] (8,2) node {$3$};
\fill [color=orange] (8.51,3.9) circle (1.5pt);
%\draw[color=orange] (8.8,3.9) node {$4$};
\fill [color=green] (8.23,5.88) circle (1.5pt);
%\draw[color=green] (8.45,6.1) node {$5$};
\fill [color=green] (6.92,7.39) circle (1.5pt);
%\draw[color=green] (7.2,7.61) node {$6$};
\fill [color=red] (5,7.96) circle (1.5pt);
%\draw[color=red] (5.13,8.25) node {$7$};
\fill [color=red] (3.08,7.39) circle (1.5pt);
%\draw[color=red] (2.9,7.61) node {$8$};
\fill [color=yellow] (1.77,5.88) circle (1.5pt);
%\draw[color=yellow] (1.6,6.1) node {$9$};
\fill [color=brown] (1.49,3.9) circle (1.5pt);
%\draw[color=brown] (1.1,3.9) node {$10$};
\fill [color=violet] (2.32,2.08) circle (1.5pt);
%\draw[color=violet] (2,2) node {$11$};
\end{scriptsize}
\end{tikzpicture}\label{colpt}}
\subfigure[Coloring $\psi_{4,15}$]{\begin{tikzpicture}[line cap=round,line join=round,>=triangle 45,x=0.5cm,y=0.5cm]
\clip(-0.5,0.5) rectangle (10.5,11.5);
\draw (6,1)-- (7.83,1.81);
\draw (9.17,3.3)-- (9.78,5.2);
\draw (9.57,7.19)-- (8.57,8.92);
\draw (6.96,10.1)-- (5,10.51);
\draw (3.04,10.1)-- (1.43,8.92);
\draw (0.43,7.19)-- (0.22,5.2);
\draw (2.17,1.81)-- (4,1);
\draw [color=blue] (4,1)-- (6,1);
\draw [color=red] (7.83,1.81)-- (9.17,3.3);
\draw [color=green] (9.78,5.2)-- (9.57,7.19);
\draw [color=red] (8.57,8.92)-- (6.96,10.1);
\draw [color=yellow] (5,10.51)-- (3.04,10.1);
\draw [color=brown] (1.43,8.92)-- (0.43,7.19);
\draw (0.43,7.19)-- (0.83,3.3);
\draw (0.83,3.3)-- (4,1);
\draw (2.17,1.81)-- (6,1);
\draw (4,1)-- (7.83,1.81);
\draw (6,1)-- (9.17,3.3);
\draw (7.83,1.81)-- (9.78,5.2);
\draw (9.17,3.3)-- (9.57,7.19);
\draw (9.78,5.2)-- (8.57,8.92);
\draw (9.57,7.19)-- (6.96,10.1);
\draw (8.57,8.92)-- (5,10.51);
\draw (6.96,10.1)-- (3.04,10.1);
\draw (5,10.51)-- (1.43,8.92);
\draw (3.04,10.1)-- (0.43,7.19);
\draw (1.43,8.92)-- (0.22,5.2);
\draw [color=violet] (0.22,5.2)-- (2.17,1.81);
\draw [color=violet] (0.22,5.2)-- (0.83,3.3);
\draw [color=violet] (0.83,3.3)-- (2.17,1.81);
\begin{scriptsize}
\fill [color=blue] (4,1) circle (1.5pt);
%\draw[color=blue] (4,0.7) node {$1$};
\fill [color=blue] (6,1) circle (1.5pt);
%\draw[color=blue] (6.13,0.7) node {$2$};
\fill [color=orange] (7.83,1.81) circle (1.5pt);
%\draw[color=orange] (8,1.5) node {$3$};
\fill [color=orange] (9.17,3.3) circle (1.5pt);
%\draw[color=orange] (9.46,3.52) node {$4$};
\fill [color=green] (9.78,5.2) circle (1.5pt);
%\draw[color=green] (10,5.41) node {$5$};
\fill [color=green] (9.57,7.19) circle (1.5pt);
%\draw[color=green] (9.8,7.41) node {$6$};
\fill [color=red] (8.57,8.92) circle (1.5pt);
%\draw[color=red] (8.8,9.14) node {$7$};
\fill [color=red] (6.96,10.1) circle (1.5pt);
%\draw[color=red] (7.2,10.31) node {$8$};
\fill [color=yellow] (5,10.51) circle (1.5pt);
%\draw[color=yellow] (5,10.9) node {$9$};
\fill [color=yellow] (3.04,10.1) circle (1.5pt);
%\draw[color=yellow] (2.7,10.31) node {$10$};
\fill [color=brown] (1.43,8.92) circle (1.5pt);
%\draw[color=brown] (1.1,9.14) node {$11$};
\fill [color=brown] (0.43,7.19) circle (1.5pt);
%\draw[color=brown] (0.1,7.41) node {$12$};
\fill [color=violet] (0.22,5.2) circle (1.5pt);
%\draw[color=violet] (-0.15,5.41) node {$13$};
\fill [color=violet] (0.83,3.3) circle (1.5pt);
%\draw[color=violet] (0.45,3.52) node {$14$};
\fill [color=violet] (2.17,1.81) circle (1.5pt);
%\draw[color=violet] (2,1.5) node {$15$};
\end{scriptsize}
\end{tikzpicture}\label{colptpt}}
\caption{Examples of $\psi_{d,n}$ in Theorem \ref{bidcn2}}
\end{figure}

\begin{theorem}\label{bidcn2}
For $d\geq3$ and $3d\leq n\leq4d-1$ we have that 
$$BI_d(C_n^2)\simeq\mathbb{S}^{3d-i-4},$$
where $n=3d+i$ for $0\leq i\leq d-1$.
\end{theorem}
\begin{proof}
For $n=3d$, $C_{3d}^2$ only has $3$ independent sets of size $d$ and these sets are disjoint, thus 
$$BI_d(C_{3d}^2)\cong\bigast_{3}\partial\Delta^{d-1}\cong\mathbb{S}^{3d-4}.$$
Assume $3d+1\leq n\leq 4d-1$. We take $p=3d-i-3$ where  $n=3d+i$ for $1\leq i\leq d-1$. Now we will define a map  $\psi_{d,n}:BI_d(C_n^2)\longrightarrow\partial\Delta^{p}$.
\begin{itemize}
    \item[$(\cdot)$] For $3d+1\leq n\leq 4d-2$ we have that $n=(3d-i-2)+2(i+1)$. We define
    $$\psi_{d,n}(j)=\left\lbrace\begin{array}{cc}
      \left\lceil\frac{j}{2}\right\rceil-1   &  \mbox{if } 1\leq j\leq4(i+1) \\
      &\\
       \left\lceil\frac{2j-4(i+1)}{2}\right\rceil-1  &  \mbox{if } j\geq4(i+1)+1
    \end{array}\right..$$
    \item[$(\cdot\cdot)$] For $n=4d-1$ we have that $n=2(2d-1)+1$. We define
    $$\psi_{d,n}(j)=\left\lbrace\begin{array}{cc}
      \left\lceil\frac{j}{2}\right\rceil-1   &  \mbox{if } 1\leq j\leq2(2d-1) \\
      &\\
      2d-2  &  \mbox{if } j=4d-1 
    \end{array}\right..$$
\end{itemize}
In both cases we are coloring the vertices of $C_n^2$ with $p+1$ colors. In the case $(\cdot)$ the first $2(i+1)$ chromatic classes have two consecutive vertices and the rest of the chromatic classes only have one 
vertex (see Figure \ref{colpt}). In the case $(\cdot\cdot)$ the first $2d-3$ chromatic classes have two consecutive members and the last one has three consecutive vertices (see Figure \ref{colptpt}). 
The only possible obstruction to these colorings to be simplicial maps is if there is a simplex with all the colors. 
We now will see that any vertex set with one color for each chromatic class induces a graph with independent number at least $d$.
For the colorings of type $(\cdot)$, let $\tau$ be a vertex set with one vertex for each color.
We take the set $S=[p]-[2i+1]$. Notice that all the chromatic classes of the colors in $S$ have only one vertex, thus 
$\psi_{d,n}^{-1}(S)=\{4i+5,\dots,3d+i\}\subseteq\tau$. Assume that in $\psi_{d,n}^{-1}(\{i+1,\dots,2i+1\})$ there is an odd vertex and let 
$i+1\leq j\leq2i+1$ the biggest color for which this happens. Then $2j+1$ is in $\tau$ and neither $2j+2$ or $2k+1$ are in $\tau$ for 
$j+1\leq k\leq 2i+1$. There are two possibilities:
\begin{itemize}
    \item[(a)] $j\equiv_20$. Then we take the vertices in $\tau$ with even color $k$ for $0\leq k\leq j$ and the vertices with odd color 
    $l$ for $j+1\leq l\leq2i+1$. These vertices form an independent set of cardinality $i+2$. We also take the vertices $4i+4+3t$ for 
    $1\leq t\leq d-i-2$, these vertices form an independent set of cardinality $d-i-2$. The union of these vertex sets form an independent set of cardinality $d$.
    \item[(b)] $j\equiv_21$. Then we take take the vertices in $\tau$ with odd color $k$ for $0\leq k\leq j$ and the vertices with even color 
    $l$ for $j+1\leq l\leq2i+1$. These vertices form an independent set of cardinality $i+2$ if $j\leq2i-1$ or of cardinality $i+1$ if $j=2i+1$. 
    For the first case we take the vertices $4i+5+3t$ for $1\leq t\leq d-i-2$ and for the second case the vertices $4i+3+3t$ for 
    $1\leq t\leq d-i-1$. Regardless of the case, the union of these vertex sets form an independent set of cardinality $d$.
\end{itemize}
Assume that for any color $i+1\leq j\leq2i+1$ the vertex $2j+2$ is in $\tau$. If $2j+2$ is in $\tau$ for some $0\leq j\leq i$, the we can 
find an independent set of cardinality $d$ contained in $\tau$ with similar arguments as before. Assume that for any color $i+1\leq j\leq2i+1$ 
the vertex $2j+1$ is in $\tau$. From this, the graph $C_n^2[\tau]$ looks like the graph in figure \ref{cn2tau}. Now, we will see that $\alpha(C_n^2[\tau])\geq d$. There are two possibilities:
\begin{itemize}
    \item[(a)] If $i\equiv_20$, we take the set
    $$\{4i+4+3t\}_{t=0}^{d-i-1}\cup\{2j+1:\;j=2l\;\mbox{ and }\;1\leq j\leq i\}\cup\{2j+2:\;j=2l+1\;\mbox{ and }\;i+1\leq j\leq2i\}$$
    where $3q+i+1=1$. 
    \item[(b)] $i\equiv_21$, we take the set
    $$\{4i+5+3t\}_{t=0}^{d-i-2}\cup\{2j+1:\;j=2l+1\;\mbox{ and }\;1\leq j\leq i\}\cup\{2j+2:\;j=2l\;\mbox{ and }\;i+1\leq j\leq2i\}.$$
\end{itemize}
In any case, the set taken is an independent set of cardinality $d$. 

Next we show that for colorings of type $(\cdot\cdot)$, any vertex set 
of cardinality $2d-1$ with all the colors induce a graph with independent number 
at least $d$. Let $\tau$ be a vertex set with one vertex for each color and take $H=C_{4d-1}^2[\tau]$. Because $H$ has order $2d-1$, if $H$ is a forest, then 
$\alpha(H)\geq\left\lceil\frac{2d-1}{2}\right\rceil=d$. We now will show that $H$ does not have a cycle. 
In $C_n^2$, for a vertex of color $j\neq2d-2$ the neighbors in a different chromatic class are of color $j-1$ or $j+1$, and 
for the color $2d-2$, the vertex $4d-2$ is the only vertex in the chromatic class that has neighbors with colors $0$ and $2d-3$. Therefore $0\leq\delta(H)\leq\Delta(H)\leq2$ as $H$ has only one vertex for each color.
Also, the only possible cycle in $H$ is of order $2d-1$ if $H\cong C_{2d-1}$. Assume that, $H\cong C_{2d-1}$. Then $4d-2$ is in $\tau$ and $N_H(4d-2)=\{1,4d-4\}$. We take 
$l=\min\{j:\;0\leq j\leq2d-3\mbox{ and }2j+2\in\tau\}$, notice that $1\leq l\leq2d-3$ because $1$ and $4d-4$ are in $\tau$. Then $2l-1$ and $2l+2$ are in $H$, but this vertices are not adjacent. Therefore $H$ is 
a forest. 

From all this we get that $BI_d(C_n^2)$ has not a simplex with all the colors regardless of the coloring.

\begin{figure}
\centering
\subfigure[{$C_n^2\left[\tau\right]$}]{\begin{tikzpicture}[line cap=round,line join=round,>=triangle 45,x=1.0cm,y=1.0cm]
\clip(2.5,1.5) rectangle (9.5,5.5);
\draw (5,2)-- (4,2);
\draw (4,2)-- (3,3);
\draw (5,2)-- (3,3);
\draw (3,3)-- (3,4);
\draw (7,2)-- (8,2);
\draw (8,2)-- (9,3);
\draw (9,3)-- (7,2);
\draw (9,3)-- (9,4);
\draw [dash pattern=on 5pt off 5pt] (3,4)-- (5,5);
\draw [dash pattern=on 5pt off 5pt] (9,4)-- (7,5);
\draw [dash pattern=on 5pt off 5pt] (5,2)-- (7,2);
\begin{scriptsize}
\fill [color=black] (3,3) circle (1.5pt);
\draw[color=black] (2.85,3) node {$1$};
\fill [color=black] (4,2) circle (1.5pt);
\draw[color=black] (4.16,1.85) node {$3d+i$};
\fill [color=black] (5,2) circle (1.5pt);
\draw[color=black] (5.2,2.28) node {$3d+i-1$};
\fill [color=black] (7,2) circle (1.5pt);
\draw[color=black] (7,2.28) node {$4i+6$};
\fill [color=black] (8,2) circle (1.5pt);
\draw[color=black] (8.14,1.8) node {$4i+5$};
\fill [color=black] (9,3) circle (1.5pt);
\draw[color=black] (8.5,3.05) node {$4i+4$};
\fill [color=black] (3,4) circle (1.5pt);
\draw[color=black] (2.85,4) node {$3$};
\fill [color=black] (9,4) circle (1.5pt);
\draw[color=black] (8.5,3.95) node {$4i+2$};
\fill [color=black] (5,5) circle (1.5pt);
\draw[color=black] (5.2,4.8) node {$2i+1$};
\fill [color=black] (7,5) circle (1.5pt);
\draw[color=black] (6.75,4.8) node {$2i+4$};
\end{scriptsize}
\end{tikzpicture}\label{cn2tau}}
\subfigure[Graph $W$ for colorings of type $(\cdot)$]{\begin{tikzpicture}[line cap=round,line join=round,>=triangle 45,x=2cm,y=2cm]
\clip(0.75,1) rectangle (4,2.75);
\draw (1.5,2.5)-- (1,2);
\draw (1,2)-- (1.25,1.5);
\draw (1.25,1.5)-- (1.75,1.25);
\draw (1.75,1.25)-- (1,2);
\draw (3.25,1.5)-- (2.75,1.25);
\draw (3.5,2)-- (3.25,1.5);
\draw (3.5,2)-- (2.75,1.25);
\draw [dash pattern=on 1pt off 1pt] (1.75,1.25)-- (2.75,1.25);
\draw [dash pattern=on 1pt off 1pt] (1.5,2.5)-- (3,2.5);
\draw (3.5,2)-- (3,2.5);
%\draw [shift={(1.5,0.5)},dash pattern=on 1pt off 1pt]  plot[domain=1.08:1.82,variable=\t]({1*1*cos(\t r)+0*1*sin(\t r)},{0*1*cos(\t r)+1*1*sin(\t r)});
%\draw [shift={(3,0.5)},dash pattern=on 1pt off 1pt]  plot[domain=1.33:2.11,variable=\t]({1*1.03*cos(\t r)+0*1.03*sin(\t r)},{0*1.03*cos(\t r)+1*1.03*sin(\t r)});
\begin{scriptsize}
\fill [color=black] (1,2) circle (1.5pt);
\draw[color=black] (1,2.12) node {$1$};
\fill [color=black] (1.5,2.5) circle (1.5pt);
\draw[color=black] (1.56,2.63) node {$2$};
\fill [color=black] (3,2.5) circle (1.5pt);
\draw[color=black] (3.14,2.63) node {$2i+1$};
\fill [color=black] (3.5,2) circle (1.5pt);
\draw[color=black] (3.65,2.12) node {$2i+2$};
\fill [color=black] (3.25,1.5) circle (1.5pt);
\draw[color=black] (3.40,1.4) node {$2i+3$};
\fill [color=black] (2.75,1.25) circle (1.5pt);
\draw[color=black] (2.9,1.15) node {$2i+4$};
\fill [color=black] (1.25,1.5) circle (1.5pt);
\draw[color=black] (1.10,1.4) node {$3d-i-2$};
\fill [color=black] (1.75,1.25) circle (1.5pt);
\draw[color=black] (1.6,1.15) node {$3d-i-3$};
\end{scriptsize}
\end{tikzpicture}\label{W}}
\caption{Auxiliary graph used in Theorem \ref{bidcn2}}
\end{figure}

Now we will see that in both cases any subgraph induced by at most 
$p$ colors have a contractible bounded independence complex. For this, first we will see that these subgraphs have less than $d$ connected components. If we take a set of colors $L=\{i_1,\dots,i_l\}$ such that 
$G[\psi_{d,n}^{-1}(L)]$ has $l$ connected components, notice that for each color $i_j$ in $L$ the colors $i_j-1,i_j+1$ are not in $L$-- where the sum is taken mod $p$. Even more, if the chromatic class of $i_j$ only has one vertex, then $i_j+2$ or $i_j-2$ or both are not in $L$. We take construct a graph $W$ with a vertex for each chromatic class of the coloring and add an edge between two vertices if between the corresponding chromatic 
classes there is an edge in $C_n^2$. Now, the disconnected subgraphs with one chromatic class for each connected component are in bijection with the independent sets of $W$, thus if we show that 
$\alpha(W)<d$, we have that $G[\psi_{d,n}^{-1}(L)]$ has less than $d$ connected components for any $L\in\mathcal{P}([p])-\{\emptyset,[p]\}$. We have two cases:
\begin{itemize}
    \item[$(\cdot)$] For the colorings of type $(\cdot)$, because of how we defined the coloring, we can take $\underline{p+1}$ as the vertex set of $W$ where the first $2(i+1)$ numbers 
    are the chromatic classes with $2$ vertices and the remaining are the chromatic classes with only one vertex. The obtained graph $W$ can be constructed by taking the graphs 
    $T_1=P_{2(i+1)},\;T_2\cong P_{3d-3i-2}^2$ and identifying the extremes of these paths (see Figure \ref{W}).  Taking $\sigma$ an independent set of size $\alpha(W)$, we have three possibilities: 
\begin{itemize}
    \item[(a)] Both $1$ and $2i+2$ are in $\sigma$. Then 
    $$|\sigma|\leq2+\alpha(T_1-N_{T_1}[1]-N_{T_2}[2i+2])+\alpha(T_2-N_{T_2}[1]-N_{T_2}[2i+2])$$
    $$=2+\left\lceil\frac{2i-2}{2}\right\rceil+\left\lceil\frac{3d-3i-8}{3}\right\rceil=d-1$$
    \item[(b)] Only one of $1$ and $2i+2$ are in $\sigma$. Assume $1$ is i $\sigma$ and $2i+2$ is not, then
    $$|\sigma|\leq1+\alpha(T_1-N_{T_1}[1]-(2i+2))+\alpha(T_2-N_{T_2}[1]-(2i+2))$$
    $$=1+\left\lceil\frac{2i-1}{2}\right\rceil+\left\lceil\frac{3d-3i-6}{3}\right\rceil=d-1$$
    \item[(c)] Non of $1$ and $2i+2$ are in $\sigma$. Then
    $$|\sigma|\leq\alpha(T_1-1-(2i+2))+\alpha(T_2-1-(2i+2))$$
    $$=\left\lceil\frac{2i}{2}\right\rceil+\left\lceil\frac{3d-3i-4}{3}\right\rceil=d-1$$
\end{itemize}
\item[$(\cdot\cdot)$] For the colorings of type $(\cdot\cdot)$ we have that $W\cong C_{2d-1}$ and we know that $\alpha(C_{2d-1})=d-1$.
\end{itemize}
Therefore $G[\psi_{d,n}^{-1}(L)]$ has less that $d$ connected components for any $L\in\mathcal{P}([p])-\{\emptyset,[p]\}$ for both types of colorings.

For $L\in\mathcal{P}([p])-\{\emptyset,[p]\}$, we take $H=G[\psi_{d,n}^{-1}(L)]$. We will now show that $BI_d(H)\simeq*$.
If $H$ has at least $2$ connected components, then $H$ is the disjoint union of powers of paths and, by Theorem \ref{theodiscnntbid}, we have that $BI_d(H)\simeq*$. 
Assume $H$ is connected. If $r=|\psi_{d,n}^{-1}(x)|\geq2$ for some $x$ in $[p]-L$, then $H$ is an induced subgraph of a graph isomorphic to $P_{n-r}^2$ and $H$ is a chordal graph. 
By Proposition \ref{propchordlbid} we have that $BI_d(H)\simeq*$. Assume that $|\psi_{d,n}^{-1}(x)|=1$ for all $x$ in $[p]-L$, notice that this implies that the coloring is of type $(\cdot)$. 
If there are two consecutive colors in $[p]-L$, then $H$ is an induced subgraph of a graph isomorphic to $P_{n-2}^2$. 
Assume there are not consecutive colors in $[p]-L$. Using Lemma \ref{lemneghbhd} we will show that $BI_d(H)\simeq BI_d(C_q)\simeq*$ where $q<2d$. 

Assume that $j,j+1,\dots,j+k\in V(H)$ and $j-1,j+k+1\notin V(H)$, where $k\geq2$. Notice that $N_H[j+1]\subseteq N_H[j+2]$ for all $k\geq2$, by Lemma \ref{lemneghbhd} we can always erase $j+1$ 
without changing the homotopy type. We affirm that we can erase $\left\lceil\frac{k-1}{2}\right\rceil$
vertices from $j,j+1,\dots,j+k$ without changing the homotopy type and in such a manner that the remaining vertices have degree $2$ in the new graph. For $k=2,3$, this is clear as we can erase $j+1$. 
Assume this is true for $4\leq l\leq k$. Assume $j,j+1,\dots,j+k+1\in V(H)$ and $j-1,j+k+1\notin V(H)$. Again, we can erase $j+1$ without changing the homotopy 
type. Then we have that $(j+2),\dots,(j+2)+(k-1)\in V(H)$ and $(j+2)-1,(j+2)+(k-1)+1\notin V(H)$, by inductive hypothesis we can erase 
$\left\lceil\frac{k-2}{2}\right\rceil+1=\left\lceil\frac{k-1}{2}\right\rceil$ vertices without changing the homotopy type and the remaining vertices have degree $2$ in the new graph.

Now, if $|[p]-L|=f$, then $H$ is obtained by erasing $f$ vertices such that any two are at distance at least two and the chromatic classes of these vertices have size $1$. These vertices also part the vertices of 
$H$ in $f$ parts of consecutive vertices as in the paragraph before. Let $n_1,\dots,n_f$ be the sizes of these parts. Then we can erase 
$$\sum_{j=1}^f\left\lceil\frac{n_j-1}{2}\right\rceil$$
vertices without changing the homotopy type and the resulting graph is a cycle of order 
$$q=3d+i-f-\sum_{j=1}^f\left\lceil\frac{n_j-1}{2}\right\rceil\leq3d+i-f-\sum_{j=1}^l\frac{n_j-1}{2}=\frac{3d+i-f}{2}\leq\frac{4d-3}{2}=2d-\frac{3}{2}<2d.$$
Therefore $BI_d(H)\simeq BI_d(C_q)\simeq*$.

Taking $\mathcal{U}=\{\Delta^{\sigma}:\;\sigma\in\partial\Delta^p-\{\emptyset\}\}$, by Theorem \ref{theopweakeqv} we have that the colorings are homotopy equivalences. 
\end{proof}

\begin{figure}
\centering
\subfigure[Coloring  $\psi_{3,16}$]{\begin{tikzpicture}[line cap=round,line join=round,>=triangle 45,x=0.5cm,y=0.5cm]
\clip(1.5,-0.5) rectangle (13.5,11.5);
\draw (10.85,1.77)-- (12.26,3.18);
\draw (13.03,7.03)-- (12.26,8.88);
\draw (9,11.05)-- (7,11.05);
\draw (3.74,8.88)-- (2.97,7.03);
\draw (3.74,3.18)-- (5.15,1.77);
\draw (2.97,7.03)-- (5.15,10.29);
\draw (2.97,7.03)-- (7,11.05);
\draw (5.15,10.29)-- (2.97,5.03);
\draw (3.74,8.88)-- (2.97,5.03);
\draw (3.74,8.88)-- (3.74,3.18);
\draw (2.97,7.03)-- (5.15,1.77);
\draw (2.97,5.03)-- (5.15,1.77);
\draw (2.97,5.03)-- (7,1);
\draw (3.74,3.18)-- (7,1);
\draw (3.74,3.18)-- (9,1);
\draw (9,1)-- (12.26,3.18);
\draw (7,1)-- (12.26,3.18);
\draw (9,1)-- (13.03,5.03);
\draw (10.85,1.77)-- (13.03,5.03);
\draw (10.85,1.77)-- (13.03,7.03);
\draw (12.26,3.18)-- (12.26,8.88);
\draw (13.03,5.03)-- (12.26,8.88);
\draw (13.03,5.03)-- (10.85,10.29);
\draw (13.03,7.03)-- (10.85,10.29);
\draw (13.03,7.03)-- (9,11.05);
\draw (12.26,8.88)-- (7,11.05);
\draw (10.85,10.29)-- (5.15,10.29);
\draw (10.85,10.29)-- (7,11.05);
\draw (9,11.05)-- (5.15,10.29);
\draw (9,11.05)-- (3.74,8.88);
\draw [color=blue] (5.15,1.77)-- (9,1);
\draw [color=blue] (5.15,1.77)-- (10.85,1.77);
\draw [color=blue] (5.15,1.77)-- (7,1);
\draw [color=blue] (7,1)-- (10.85,1.77);
\draw [color=blue] (7,1)-- (9,1);
\draw [color=blue] (9,1)-- (10.85,1.77);
\draw [color=orange] (12.26,3.18)-- (13.03,5.03);
\draw [color=orange] (13.03,5.03)-- (13.03,7.03);
\draw [color=orange] (13.03,7.03)-- (12.26,3.18);
\draw [color=green] (12.26,8.88)-- (9,11.05);
\draw [color=green] (9,11.05)-- (10.85,10.29);
\draw [color=green] (10.85,10.29)-- (12.26,8.88);
\draw [color=red] (3.74,8.88)-- (5.15,10.29);
\draw [color=red] (5.15,10.29)-- (7,11.05);
\draw [color=red] (7,11.05)-- (3.74,8.88);
\draw [color=yellow] (2.97,7.03)-- (3.74,3.18);
\draw [color=yellow] (3.74,3.18)-- (2.97,5.03);
\draw [color=yellow] (2.97,5.03)-- (2.97,7.03);
\begin{scriptsize}
\fill [color=blue] (7,1) circle (1.5pt);
%\draw[color=blue] (7.16,0.68) node {$2$};
\fill [color=blue] (9,1) circle (1.5pt);
%\draw[color=blue] (9.16,0.68) node {$3$};
\fill [color=blue] (10.85,1.77) circle (1.5pt);
%\draw[color=blue] (11,1.45) node {$4$};
\fill [color=orange] (12.26,3.18) circle (1.5pt);
%\draw[color=orange] (12.58,3.18) node {$5$};
\fill [color=orange] (13.03,5.03) circle (1.5pt);
%\draw[color=orange] (13.38,5.03) node {$6$};
\fill [color=orange] (13.03,7.03) circle (1.5pt);
%\draw[color=orange] (13.38,7.03) node {$7$};
\fill [color=green] (12.26,8.88) circle (1.5pt);
%\draw[color=green] (12.42,9.16) node {$8$};
\fill [color=green] (10.85,10.29) circle (1.5pt);
%\draw[color=green] (11,10.56) node {$9$};
\fill [color=green] (9,11.05) circle (1.5pt);
%\draw[color=green] (9.12,11.34) node {$10$};
\fill [color=red] (7,11.05) circle (1.5pt);
%\draw[color=red] (6.74,11.34) node {$11$};
\fill [color=red] (5.15,10.29) circle (1.5pt);
%\draw[color=red] (4.8,10.56) node {$12$};
\fill [color=red] (3.74,8.88) circle (1.5pt);
%\draw[color=red] (3.48,9.16) node {$13$};
\fill [color=yellow] (2.97,7.03) circle (1.5pt);
%\draw[color=brown] (2.55,7.03) node {$14$};
\fill [color=yellow] (2.97,5.03) circle (1.5pt);
%\draw[color=brown] (2.55,5.03) node {$15$};
\fill [color=yellow] (3.74,3.18) circle (1.5pt);
%\draw[color=brown] (3.33,3.18) node {$16$};
\fill [color=blue] (5.15,1.77) circle (1.5pt);
%\draw[color=blue] (5.3,1.45) node {$1$};
\end{scriptsize}
\end{tikzpicture}}
\subfigure[Coloring $\psi_{4,23}$]{\begin{tikzpicture}[line cap=round,line join=round,>=triangle 45,x=0.4cm,y=0.4cm]
\clip(1,0.5) rectangle (17,16);
\draw[color=blue] (8,1)-- (10,1);
\draw[color=blue] (10,1)-- (11.93,1.54);
\draw[color=blue] (11.93,1.54)-- (13.63,2.58);
\draw (13.63,2.58)-- (15,4.04);
\draw[color=orange] (15,4.04)-- (15.92,5.82);
\draw[color=orange] (15.92,5.82)-- (16.33,7.77);
\draw[color=orange] (16.33,7.77)-- (16.19,9.77);
\draw (16.19,9.77)-- (15.52,11.65);
\draw[color=green] (15.52,11.65)-- (14.37,13.29);
\draw[color=green] (14.37,13.29)-- (12.82,14.55);
\draw (12.82,14.55)-- (10.98,15.35);
\draw[color=red] (10.98,15.35)-- (9,15.62);
\draw[color=red] (9,15.62)-- (7.02,15.35);
\draw (7.02,15.35)-- (5.18,14.55);
\draw[color=yellow] (5.18,14.55)-- (3.63,13.29);
\draw[color=yellow] (3.63,13.29)-- (2.48,11.65);
\draw (2.48,11.65)-- (1.81,9.77);
\draw[color=brown] (1.81,9.77)-- (1.67,7.77);
\draw[color=brown] (1.67,7.77)-- (2.08,5.82);
\draw (2.08,5.82)-- (3,4.04);
\draw[color=violet] (3,4.04)-- (4.37,2.58);
\draw[color=violet] (4.37,2.58)-- (6.07,1.54);
\draw (6.07,1.54)-- (8,1);
\draw (10,1)-- (15,4.04);
\draw (11.93,1.54)-- (15,4.04);
\draw (11.93,1.54)-- (15.92,5.82);
\draw (13.63,2.58)-- (15.92,5.82);
\draw (13.63,2.58)-- (16.33,7.77);
\draw (15.92,5.82)-- (15.52,11.65);
\draw (16.33,7.77)-- (15.52,11.65);
\draw (16.33,7.77)-- (14.37,13.29);
\draw (16.19,9.77)-- (14.37,13.29);
\draw (16.19,9.77)-- (12.82,14.55);
\draw (15.52,11.65)-- (10.98,15.35);
\draw (14.37,13.29)-- (10.98,15.35);
\draw (14.37,13.29)-- (9,15.62);
\draw (12.82,14.55)-- (9,15.62);
\draw (12.82,14.55)-- (7.02,15.35);
\draw (10.98,15.35)-- (5.18,14.55);
\draw (9,15.62)-- (5.18,14.55);
\draw (9,15.62)-- (3.63,13.29);
\draw (7.02,15.35)-- (3.63,13.29);
\draw (7.02,15.35)-- (2.48,11.65);
\draw (5.18,14.55)-- (1.81,9.77);
\draw (3.63,13.29)-- (1.81,9.77);
\draw (3.63,13.29)-- (1.67,7.77);
\draw (2.48,11.65)-- (1.67,7.77);
\draw (2.48,11.65)-- (2.08,5.82);
\draw (1.81,9.77)-- (3,4.04);
\draw (1.67,7.77)-- (4.37,2.58);
\draw (1.67,7.77)-- (3,4.04);
\draw (2.08,5.82)-- (4.37,2.58);
\draw (2.08,5.82)-- (6.07,1.54);
\draw (3,4.04)-- (8,1);
\draw (4.37,2.58)-- (8,1);
\draw (4.37,2.58)-- (10,1);
\draw (6.07,1.54)-- (10,1);
\draw (6.07,1.54)-- (11.93,1.54);
\draw [color=blue] (8,1)-- (13.63,2.58);
\draw [color=blue] (8,1)-- (11.93,1.54);
\draw [color=blue] (10,1)-- (13.63,2.58);
\draw [color=orange] (15,4.04)-- (16.19,9.77);
\draw [color=orange] (15,4.04)-- (16.33,7.77);
\draw [color=orange] (15.92,5.82)-- (16.19,9.77);
\draw [color=green] (15.52,11.65)-- (12.82,14.55);
\draw [color=red] (10.98,15.35)-- (7.02,15.35);
\draw [color=yellow] (5.18,14.55)-- (2.48,11.65);
\draw [color=brown] (1.81,9.77)-- (2.08,5.82);
\draw [color=violet] (3,4.04)-- (6.07,1.54);
\begin{scriptsize}
\fill [color=blue] (8,1) circle (1.5pt);
%\draw[color=blue] (8.16,1.28) node {$A$};
\fill [color=blue] (10,1) circle (1.5pt);
%\draw[color=blue] (10.16,1.28) node {$B$};
\fill [color=blue] (11.93,1.54) circle (1.5pt);
%\draw[color=blue] (12.08,1.82) node {$C$};
\fill [color=blue] (13.63,2.58) circle (1.5pt);
%\draw[color=blue] (13.8,2.86) node {$D$};
\fill [color=orange] (15,4.04) circle (1.5pt);
%\draw[color=orange] (15.14,4.32) node {$E$};
\fill [color=orange] (15.92,5.82) circle (1.5pt);
%\draw[color=orange] (16.06,6.1) node {$F$};
\fill [color=orange] (16.33,7.77) circle (1.5pt);
%\draw[color=orange] (16.48,8.06) node {$G$};
\fill [color=orange] (16.19,9.77) circle (1.5pt);
%\draw[color=orange] (16.36,10.04) node {$H$};
\fill [color=green] (15.52,11.65) circle (1.5pt);
%\draw[color=green] (15.64,11.94) node {$I$};
\fill [color=green] (14.37,13.29) circle (1.5pt);
%\draw[color=green] (14.46,13.56) node {$J$};
\fill [color=green] (12.82,14.55) circle (1.5pt);
%\draw[color=green] (12.96,14.84) node {$K$};
\fill [color=red] (10.98,15.35) circle (1.5pt);
%\draw[color=red] (11.12,15.62) node {$L$};
\fill [color=red] (9,15.62) circle (1.5pt);
%\draw[color=red] (9.18,15.9) node {$M$};
\fill [color=red] (7.02,15.35) circle (1.5pt);
%\draw[color=red] (7.18,15.62) node {$N$};
\fill [color=yellow] (5.18,14.55) circle (1.5pt);
%\draw[color=yellow] (5.34,14.84) node {$O$};
\fill [color=yellow] (3.63,13.29) circle (1.5pt);
%\draw[color=yellow] (3.78,13.56) node {$P$};
\fill [color=yellow] (2.48,11.65) circle (1.5pt);
%\draw[color=yellow] (2.64,11.94) node {$Q$};
\fill [color=brown] (1.81,9.77) circle (1.5pt);
%\draw[color=brown] (1.94,10.04) node {$R$};
\fill [color=brown] (1.67,7.77) circle (1.5pt);
%\draw[color=brown] (1.82,8.06) node {$S$};
\fill [color=brown] (2.08,5.82) circle (1.5pt);
%\draw[color=brown] (2.22,6.1) node {$T$};
\fill [color=violet] (3,4.04) circle (1.5pt);
%\draw[color=violet] (3.16,4.32) node {$U$};
\fill [color=violet] (4.37,2.58) circle (1.5pt);
%\draw[color=violet] (4.5,2.86) node {$V$};
\fill [color=violet] (6.07,1.54) circle (1.5pt);
%\draw[color=violet] (6.26,1.82) node {$W$};
\end{scriptsize}
\end{tikzpicture}}
\caption{Examples of $\psi_{d,n}$ in Theorem \ref{cyclpwonr}}\label{colr3d16}
\end{figure}

\begin{theorem}\label{cyclpwonr}
For $d\geq3$, $r\geq3$ and $2rd-(r-1)\leq n\leq2rd-1$ we have that 
$$BI_d(C_n^r)\simeq\mathbb{S}^{2d-3}.$$
\end{theorem}
\begin{proof}
The proof is similar to the proof of Theorem \ref{bidcn2}. For $2rd-(r-1)\leq n\leq2rd-1$ we have that 
$n=r(2d-1)+i$ with $1\leq i\leq r-1$. We take $p=2d-2$ and define the coloring $\psi_{d,n}:BI_d(C_n^r)\longrightarrow\partial\Delta^p$ as:
$$\psi_{d,n}(j)=\left\lbrace\begin{array}{cc}
   \left\lceil\frac{j}{r+1}\right\rceil-1  & \mbox{ if } 1\leq j \leq(r+1)i\\
   &\\
    \left\lceil\frac{j-(r+1)i}{r}\right\rceil+i-1  & \mbox{ if } (r+1)i+1\leq j
\end{array}\right..$$
The idea of the coloring is that the first $i$ chromatic classes have $r+1$ vertices and the rest have $r$ (see Figure \ref{colr3d16}).

If a vertex $i$ has color $j$ then $N_{C_n^r}(i)\subseteq\psi_{d,n}^{-1}(\{j-1,j,j+1\})$. Any heterochromatic vertex set of size $2d-1$ induces a forest. Assume this is not true and let 
$\tau$ be a heterochromatic vertex set of size $2d-1$ such that $C_n^r[\tau]$ is not a forest. Then $C_n^r[\tau]\cong C_{2d-1}$ and $V(C_n^r)-\tau$ can be partition in $2d-1$ sets of consecutive vertices, each 
set of size at most $r-1$. Then
$$n\leq (r-1)(2d-1)+2d-1=2rd-r<2rd-(r-1)\leq n.$$
Therefore $C_n^r[\tau]$ is a forest and $\alpha(C_n^r[\tau])\geq d$. From this we get that the coloring defines a simplicial map. 

Now we will see that for any $S\in\mathcal{P}([p])-\{[p],\emptyset\}$, the subgraph $C_n^r[\psi_{d,n}^{-1}(S)]$ has less than $d$ connected components. We define a graph $W$ as the one in the proof of Theorem 
\ref{bidcn2} and again the independent sets of $W$ are in bijection with the disconnected subgraphs $C_n^r[\psi_{d,n}^{-1}(S)]$ with one color for each connected component. 
Notice that $W\cong C_{2d-1}$ and $\alpha(W)=d-1$. Therefore $C_n^r[\psi_{d,n}^{-1}(S)]$ has less than $d$ connected components for any $S\in\mathcal{P}([p])-\{[p],\emptyset\}$ and 
$BI_d(C_n^r[\psi_{d,n}^{-1}(S)])\simeq*$ by Theorem \ref{theodiscnntbid}, as each component is isomorphic to the $r$th power of a path.

Taking $\mathcal{U}=\{\Delta^{\sigma}:\;\sigma\in\partial\Delta^p-\{\emptyset\}\}$, by Theorem \ref{theopweakeqv} we have that the colorings are homotopy equivalences.
\end{proof}

\begin{theorem}\label{thmtotcyclpwrs}
(a) For $r\geq2$, $n\geq 2rd$ and $d\geq2$, we have that
$$\Delta_d^t(C_n^p)\simeq\mathbb{S}^{n-2d}$$
for any $1\leq p\leq r$.

(b) For $d\geq3$ and $3d\leq n\leq3d+i$ with $0\leq i\leq d-1$ we have that
$$\Delta_d^t(C_n^2)\simeq\mathbb{S}^{2i+1}.$$

(c) For $r\geq3$, $d\geq3$ and $2rd-(r-1)\leq n\leq2rd-1$ we have that
$$\Delta_d^t(C_n^p)\simeq\mathbb{S}^{n-2d}.$$
\end{theorem}
\begin{proof}
By Theorems \ref{cyclepwerbid}, \ref{bidcn2} and \ref{cyclpwonr}, in the three cases the complexes have the homology of a sphere. Therefore, by Theorems \ref{dualidadalexander} and \ref{homttyphomlgrp}, 
we only need to show that the compexes are simply connected in the three cases.

\textit{(a)} %By Theorems \ref{dualidadalexander}, \ref{homttyphomlgrp} and \ref{cyclepwerbid}, we only need to show that $\Delta_d^t(C_n^p)$ is simply connected. 
Notice that $C_n^p-N_{C_n^p}[1]\cong P_{n-2p-1}^p$ where $n-2p-1\geq2rd-2p-1\geq2r(d-1)-1$, thus
$$\alpha(C_n^p-N_{C_n^p}[1])\geq\left\lceil\frac{2r(d-1)-1}{r+1}\right\rceil\geq d-1.$$
By Proposition \ref{homosusp}
$$\Delta_d^t(C_n^p)\simeq\Sigma\left(\Delta_d^t(C_n^p-1)\cap\left(\Delta^{N_{C_n^p}[1]}*\Delta_{d-1}^t\left((C_n^p-N_{C_n}[1])\right)\right)\right)$$
Now, $S=\{2,3+p,\dots,d+1+(d-1)p\}$ is an independent set such that $S\cap N_{C_n^p}(1)=\{2\}$, thus there is a non-empty simplex in the intersection.
If the intersection is disconnected, then $\Delta_d^t(C_n^p)$ has the homotopy type of a wedge of spaces and at least one of these
spaces is a circle, thus $\tilde{H}_1(\Delta_d^t(C_n^p))\ncong0$. But this is not possible as by Alexander Duality and Theorem \ref{cyclepwerbid} the 
only non-trivial reduced homology group is for $n-2d\geq2d(r-1)>1$. Therefore the intersection is connected and $\Delta_d^t(C_n^p)$ 
is simply connected.

\textit{(b)} The case $n=3d$ is proved in Proposition \ref{propcr+1d}. For $3d+1\leq n\leq4d-1$, with the same arguments as before we can show that $\Delta_d^t(C_n^2)$ is 
simply connected. 
%By Theorems \ref{dualidadalexander}, \ref{homttyphomlgrp} and \ref{bidcn2} we have that 
%$\Delta_d^t(C_n^2)\simeq\mathbb{S}^{2i+1}.$

\textit{(c)} Notice that $C_n^r-N_{C_n^r}[1]\cong P_k^r$ where $k\geq2rd-(r-1)-(2r+1)=2r(d-1)-r$ and 
$$\alpha(P_k^r)\geq\left\lceil\frac{2r(d-1)-r}{r+1}\right\rceil=\left\lceil d-1+\frac{(r-1)(d-1)-r}{r+1}\right\rceil.$$
Because $d\geq3$ we have that $(r-1)(d-1)-r>0$, thus $\alpha(P_k^r)\geq d-1$.  We can proceed as in point (a).
\end{proof}

\begin{prop}\label{propcr+1d}
If $n=(r+1)d$, then $\Delta_d^t\left(C_{n}^r\right)\simeq\mathbb{S}^{r-1}$
\end{prop}
\begin{proof}
Because $\alpha\left(C_n^r\right)=\left\lfloor\frac{n}{r+1}\right\rfloor$, for $n=(r+1)d$ there are $r+1$ independent sets of 
size $d$ which are disjoint. Let $S_1,\dots,S_{r+1}$ be these independent sets, then
$$\Delta_d^t\left(C_{n}^r\right)=\Delta^{V-S_1}\cup\cdots\cup\Delta^{V-S_{r+1}}$$
The intersection of any $r$ or less of these complexes is contractible, and the intersection of all of the complexes is empty. Therefore, 
by the Nerve Theorem \ref{nrvthm}, we have that
$$\Delta_d^t\left(C_{n}^r\right)\simeq\mathbb{S}^{r-1}.$$
\end{proof}

\begin{cor}
For $d\geq3$, if $n=(r+1)d$, then
$$BI_d\left(C_{n}^r\right)\simeq\mathbb{S}^{(r+1)(d-2)+r}$$
\end{cor}
Now we focus on $\Delta_2^t(C_n^r)$.
\begin{lem}\label{theoconncyclepower}
$\pi_1\left(\Delta_2^t(C_n^r)\right)\cong0$ for $n\geq2r+3$ and $r\geq3$.
\end{lem}
\begin{proof}
We will see that $\mathrm{sk}_2\Delta_2^t(C_n^r)=\mathrm{sk}_2\Delta^{\underline{n}}$. 
Take $\{x,y,z\}$ in $\binom{\underline{n}}{3}$. If $C_n^r-\{x,y,z\}$ does not have an independent set of cardinality $2$, then 
$C_n^r-\{x,y,z\}\cong K_{n-3}$, but this is not possible because in $C_n^r$ the biggest complete subgraph possible is $K_{r+1}$. 
\end{proof}

\begin{prop}\label{propclcipow}\citep{clicomgrppowers}
For any $n\geq3$ and $\displaystyle 0\leq r<\frac{n}{2}$
$$BI_2(C_n^r)\simeq\left\lbrace\begin{array}{cc}\displaystyle
   \bigvee_{n-2r-1}\mathbb{S}^{2l} & \mbox{ if } r=\frac{l}{2l+1}n \\
    & \\
    \mathbb{S}^{2l+1} & \mbox{ if } \frac{l}{2l+1}n<r<\frac{l+1}{2l+3}n
\end{array}
\right.\mbox{ for some } l\geq0$$
\end{prop}

As an immediate consequence of the last proposition we have that $BI_2(C_n^r)\simeq\mathbb{S}^1$ for $n\geq3r+1$.

\begin{theorem}
For $r\geq3$ we have the following 
\begin{itemize}
\item[(a)] for $n=2r+2$
$$\Delta_2^t(C_n^r)\simeq\mathbb{S}^{r-1}$$
 \item[(b)] for $2r+3\leq n\leq3r-1$
$$\Delta_2^t(C_n^r)\simeq\left\lbrace\begin{array}{cc}
   \displaystyle\bigvee_{n-2r-1}\mathbb{S}^{n-2l-3}  &  \mbox{ if } r=\frac{l}{2l+1}n \\
    & \\
   \mathbb{S}^{n-2l-4}  &  \mbox{ if } \frac{l}{2l+1}n<r<\frac{l+1}{2l+3}n
\end{array}
\right. \mbox{ for some } l\geq0$$
\item[(c)] for $n=3r$
$$\Delta_2^t(C_n^r)\simeq\bigvee_{r-1}\mathbb{S}^{3r-5}$$
\item[(d)] for $n\geq3r+1$
$$\Delta_2^t(C_n^r)\simeq\mathbb{S}^{n-4}$$
\end{itemize}
\end{theorem}
\begin{proof}
The point (a) is a particular case of Proposition \ref{propcr+1d}. The rest follows by Proposition \ref{propclcipow}, Lemma \ref{theoconncyclepower} and, Theorems \ref{homttyphomlgrp} and \ref{dualidadalexander}.
\end{proof}

\begin{que}
Which are the homotopy types of $BI_d(C_n^r)$ and $\Delta_d^t(C_n^r)$ for $d\geq3$, $r\geq3$ and $(r+1)d+1\leq n\leq2rd-r$?
\end{que}

\subsection{Somme Cartesian products}
We take the lattice $L(n_1,\dots,n_k)\cong P_{n_1}\oblong\cdots\oblong P_{n_k}$, because this graph does not have triangles, we have that 
$$BI_2(L(n_1,\dots,n_k))\simeq\bigvee_{s_k(n_1,\dots,n_k)}\mathbb{S}^1$$
where 
$$s_k(n_1,\dots,n_k)=\left(\sum_{i=1}^k(n_i-1)\prod_{j\neq i}n_j\right)-\prod_{i=1}^kn_i +1$$
is the number of edges in the graph minus the number of edges in a spanning tree, \textit{i.e.} the cyclomatic number of the graph.
\begin{theorem}\label{teolattice}
For $n_1,\dots,n_k$ integers equal or bigger than $2$, we have that
$$\Delta_2^t(L(n_1,\dots,n_k))\simeq\bigvee_{s_k(n_1,\dots,n_k)}\mathbb{S}^{n-4}.$$
\end{theorem}
\begin{proof}
For $k=2$ and $n_1=n_2=2$ it is clear from Theorem \ref{totcomcycles}. For the other cases we take three distinct vertices $x,y,z$. If $k\geq3$, notice that in the cube induced by the vertices 
$\underline{2}^3\times\{1\}\times\cdots\times\{1\}$ the independent sets 
$\{(1,1,1,\dots,1)$, $(2,2,1,\dots,1)\}$, $\{(2,1,1,\dots,1)$, $(1,2,1,\dots,1)\}$, $\{(1,1,2,\dots,1)$, $(2,2,2,\dots,1)\}$, 
$\{(2,1,2,\dots,1),\dots,(1,2,2,\dots,1)\}$ are disjoint, therefore $\{x,y,z\}$ is a simplex of 
$\Delta_2^t(L(n_1,\dots,n_k))$. Lastly, if $k=2$ and $n_1\geq3$. We take $H$ the subgraph induced by the vertex set $\underline{3}\times\underline{2}$ (see Figure \ref{lat32}). 
If $\sigma$ is contained in $V(H)$, there are up to isomorphism three cases for $H[V(H)-\sigma]$ seen in Figures \ref{case1},\ref{case2},\ref{case3}, where the red vertices are $\sigma$, and we have that 
$\alpha(H[V(H)-\sigma])=2$. Then  $\{x,y,z\}$ is a simplex of $\Delta_2^t(L(n_1,n_2))$. If $\sigma$ is not contained in $H$, then 
$\alpha(H[V(H)-\sigma])\geq2$ and  $\{x,y,z\}$ is a simplex of $\Delta_2^t(L(n_1,n_2))$. Regardless of the case we have that 
$\mathrm{sk}_2\Delta_2^t(L(n_1,\dots,n_k))=\mathrm{sk}_2\Delta^{\underline{n_1}\times\cdots\times\underline{n_k}}$.
Therefore $\Delta_2^t(L(n_1,\dots,n_k))$ is simply connected. 
By Theorems \ref{homttyphomlgrp}, \ref{dualidadalexander}, \ref{gradconse} and the observation previous to the theorem, we have that $\Delta_2^t(L(n_1,\dots,n_k))$ has the expected homotopy type.
\end{proof}

\begin{figure}
\centering
\subfigure[$H$]{\begin{tikzpicture}[line cap=round,line join=round,>=triangle 45,x=1.0cm,y=1.0cm]
\clip(1.5,0.5) rectangle (4.5,2.5);
\draw (2,2)-- (3,2);
\draw (3,2)-- (4,2);
\draw (4,2)-- (4,1);
\draw (4,1)-- (3,1);
\draw (3,1)-- (3,2);
\draw (3,1)-- (2,1);
\draw (2,1)-- (2,2);
\begin{scriptsize}
\fill [color=black] (2,1) circle (1.5pt);
\fill [color=black] (2,2) circle (1.5pt);
\fill [color=black] (3,2) circle (1.5pt);
\fill [color=black] (4,2) circle (1.5pt);
\fill [color=black] (4,1) circle (1.5pt);
\fill [color=black] (3,1) circle (1.5pt);
\end{scriptsize}
\end{tikzpicture}\label{lat32}}
\subfigure[Case 1]{\begin{tikzpicture}[line cap=round,line join=round,>=triangle 45,x=1.0cm,y=1.0cm]
\clip(1.5,0.5) rectangle (4.5,2.5);
\draw (2,2)-- (3,2);
\draw (3,2)-- (4,2);
\draw (4,2)-- (4,1);
\draw (4,1)-- (3,1);
\draw (3,1)-- (3,2);
\draw (3,1)-- (2,1);
\draw (2,1)-- (2,2);
\begin{scriptsize}
\fill [color=red] (2,1) circle (1.5pt);
\fill [color=red] (2,2) circle (1.5pt);
\fill [color=red] (3,2) circle (1.5pt);
\fill [color=black] (4,2) circle (1.5pt);
\fill [color=black] (4,1) circle (1.5pt);
\fill [color=black] (3,1) circle (1.5pt);
\end{scriptsize}
\end{tikzpicture}\label{case1}}
\subfigure[Case 2]{\begin{tikzpicture}[line cap=round,line join=round,>=triangle 45,x=1.0cm,y=1.0cm]
\clip(1.5,0.5) rectangle (4.5,2.5);
\draw (2,2)-- (3,2);
\draw (3,2)-- (4,2);
\draw (4,2)-- (4,1);
\draw (4,1)-- (3,1);
\draw (3,1)-- (3,2);
\draw (3,1)-- (2,1);
\draw (2,1)-- (2,2);
\begin{scriptsize}
\fill [color=red] (2,1) circle (1.5pt);
\fill [color=black] (2,2) circle (1.5pt);
\fill [color=red] (3,2) circle (1.5pt);
\fill [color=black] (4,2) circle (1.5pt);
\fill [color=black] (4,1) circle (1.5pt);
\fill [color=red] (3,1) circle (1.5pt);
\end{scriptsize}
\end{tikzpicture}\label{case2}}
\subfigure[Case 3]{\begin{tikzpicture}[line cap=round,line join=round,>=triangle 45,x=1.0cm,y=1.0cm]
\clip(1.5,0.5) rectangle (4.5,2.5);
\draw (2,2)-- (3,2);
\draw (3,2)-- (4,2);
\draw (4,2)-- (4,1);
\draw (4,1)-- (3,1);
\draw (3,1)-- (3,2);
\draw (3,1)-- (2,1);
\draw (2,1)-- (2,2);
\begin{scriptsize}
\fill [color=black] (2,1) circle (1.5pt);
\fill [color=red] (2,2) circle (1.5pt);
\fill [color=red] (3,2) circle (1.5pt);
\fill [color=red] (4,2) circle (1.5pt);
\fill [color=black] (4,1) circle (1.5pt);
\fill [color=black] (3,1) circle (1.5pt);
\end{scriptsize}
\end{tikzpicture}\label{case3}}
\caption{Auxiliary graphs for Theorem \ref{teolattice}}
\end{figure}

\begin{que}
Which are the homotopy types of $BI_d(L(n_1,\dots,n_k))$ and $\Delta_d^t(L(n_1,\dots,n_k))$ for $d\geq3$?
\end{que}

We take $K(n_1,n_2,\dots,n_k)=K_{n_1}\oblong K_{n_2}\oblong\cdots\oblong K_{n_k}$. In \citep{totcutcompl} the homotopy type of 
$\Delta_2^t(K(n,2))$ was determined, more precisely $\Delta_2^t(K(n,2))$ has the homotopy type of a wedge of $(2n-4)$-spheres if $n\geq2$ \citep[see Theorem 4.8]{totcutcompl}. 
We extend this by determining the homotopy type of $\Delta_2^t(K(n_1,n_2,\dots,n_k))$ for all $k\geq2$. First we define some functions. For  
$k\geq2$ and $n_1,\dots,n_k$ positive integers we take
$$F_k(n_1,\dots,n_,n_k)=(k-1)n+1-\sum_{i=1}^k\prod_{j\neq i}n_i$$
where $n=n_1\cdots n_K$. For $k\geq3$ and $n_1,\dots,n_k\geq2$ it is not hard to prove that these functions have the following properties:
$$F_k(1,\dots,1)=0,\;\;F_k(2,\dots,2)=k2^{k-1}-2^k+1,$$
$$F_k(n_1,\dots,n_{k-1},1)=F_{k-1}(n_1,\dots,n_{k-1}),$$
$$F_k(n_1,\dots,n_k)=F_k(n_1,\dots,n_k-1)+F_{k-1}(n_1,\dots,n_{k-1})+n_1\cdots n_{k-1}-1.$$
\begin{figure}
\centering
\subfigure[Case $k=2$]{\begin{tikzpicture}[line cap=round,line join=round,>=triangle 45,x=1.0cm,y=1.0cm]
\clip(3.5,1.5) rectangle (6.5,4.5);
\draw (4,4)-- (5,4);
\draw (5,4)-- (6,4);
\draw (4,3)-- (5,3);
\draw (5,3)-- (6,3);
\draw (4,2)-- (5,2);
\draw (5,2)-- (6,2);
\draw (4,2)-- (4,3);
\draw (4,3)-- (4,4);
\draw (5,4)-- (5,3);
\draw (5,3)-- (5,2);
\draw (6,2)-- (6,3);
\draw (6,3)-- (6,4);
\draw (4,4) .. controls (4.5,4.5) and (5.5,4.5) .. (6,4);
\draw (4,3) .. controls (4.5,3.5) and (5.5,3.5) .. (6,3);
\draw (4,2) .. controls (4.5,2.5) and (5.5,2.5) .. (6,2);
\draw (4,4) .. controls (4.5,3.5) and (4.5,2.5) .. (4,2);
\draw (5,4) .. controls (5.5,3.5) and (5.5,2.5) .. (5,2);
\draw (6,4) .. controls (6.5,3.5) and (6.5,2.5) .. (6,2);
\begin{scriptsize}
\fill [color=black] (4,4) circle (1.5pt);
\fill [color=black] (5,4) circle (1.5pt);
\fill [color=black] (6,4) circle (1.5pt);
\fill [color=black] (4,3) circle (1.5pt);
\fill [color=black] (5,3) circle (1.5pt);
\fill [color=black] (6,3) circle (1.5pt);
\fill [color=black] (4,2) circle (1.5pt);
\fill [color=black] (5,2) circle (1.5pt);
\fill [color=black] (6,2) circle (1.5pt);
\end{scriptsize}
\end{tikzpicture}\label{grith3x3}} 
\subfigure[Case $k=2$ (a)]{\begin{tikzpicture}[line cap=round,line join=round,>=triangle 45,x=1.0cm,y=1.0cm]
\clip(3.5,1.5) rectangle (6.5,4.5);
\draw (4,4)-- (5,4);
\draw (5,4)-- (6,4);
\draw (4,3)-- (5,3);
\draw (5,3)-- (6,3);
\draw (4,2)-- (5,2);
\draw (5,2)-- (6,2);
\draw (4,2)-- (4,3);
\draw (4,3)-- (4,4);
\draw (5,4)-- (5,3);
\draw (5,3)-- (5,2);
\draw (6,2)-- (6,3);
\draw (6,3)-- (6,4);
\draw (4,4) .. controls (4.5,4.5) and (5.5,4.5) .. (6,4);
\draw (4,3) .. controls (4.5,3.5) and (5.5,3.5) .. (6,3);
\draw (4,2) .. controls (4.5,2.5) and (5.5,2.5) .. (6,2);
\draw (4,4) .. controls (4.5,3.5) and (4.5,2.5) .. (4,2);
\draw (5,4) .. controls (5.5,3.5) and (5.5,2.5) .. (5,2);
\draw (6,4) .. controls (6.5,3.5) and (6.5,2.5) .. (6,2);
\begin{scriptsize}
\fill [color=red] (4,4) circle (1.5pt);
\fill [color=red] (5,4) circle (1.5pt);
\fill [color=red] (6,4) circle (1.5pt);
\fill [color=black] (4,3) circle (1.5pt);
\fill [color=black] (5,3) circle (1.5pt);
\fill [color=black] (6,3) circle (1.5pt);
\fill [color=black] (4,2) circle (1.5pt);
\fill [color=black] (5,2) circle (1.5pt);
\fill [color=black] (6,2) circle (1.5pt);
\end{scriptsize}
\end{tikzpicture}\label{grith3x3a}} 
\subfigure[Case $k=2$ (b)]{\begin{tikzpicture}[line cap=round,line join=round,>=triangle 45,x=1.0cm,y=1.0cm]
\clip(3.5,1.5) rectangle (6.5,4.5);
\draw (4,4)-- (5,4);
\draw (5,4)-- (6,4);
\draw (4,3)-- (5,3);
\draw (5,3)-- (6,3);
\draw (4,2)-- (5,2);
\draw (5,2)-- (6,2);
\draw (4,2)-- (4,3);
\draw (4,3)-- (4,4);
\draw (5,4)-- (5,3);
\draw (5,3)-- (5,2);
\draw (6,2)-- (6,3);
\draw (6,3)-- (6,4);
\draw (4,4) .. controls (4.5,4.5) and (5.5,4.5) .. (6,4);
\draw (4,3) .. controls (4.5,3.5) and (5.5,3.5) .. (6,3);
\draw (4,2) .. controls (4.5,2.5) and (5.5,2.5) .. (6,2);
\draw (4,4) .. controls (4.5,3.5) and (4.5,2.5) .. (4,2);
\draw (5,4) .. controls (5.5,3.5) and (5.5,2.5) .. (5,2);
\draw (6,4) .. controls (6.5,3.5) and (6.5,2.5) .. (6,2);
\begin{scriptsize}
\fill [color=red] (4,4) circle (1.5pt);
\fill [color=red] (5,4) circle (1.5pt);
\fill [color=black] (6,4) circle (1.5pt);
\fill [color=red] (4,3) circle (1.5pt);
\fill [color=black] (5,3) circle (1.5pt);
\fill [color=black] (6,3) circle (1.5pt);
\fill [color=black] (4,2) circle (1.5pt);
\fill [color=black] (5,2) circle (1.5pt);
\fill [color=black] (6,2) circle (1.5pt);
\end{scriptsize}
\end{tikzpicture}\label{grith3x3b}}
\subfigure[Case $k=2$ (c)]{\begin{tikzpicture}[line cap=round,line join=round,>=triangle 45,x=1.0cm,y=1.0cm]
\clip(3.5,1.5) rectangle (6.5,4.5);
\draw (4,4)-- (5,4);
\draw (5,4)-- (6,4);
\draw (4,3)-- (5,3);
\draw (5,3)-- (6,3);
\draw (4,2)-- (5,2);
\draw (5,2)-- (6,2);
\draw (4,2)-- (4,3);
\draw (4,3)-- (4,4);
\draw (5,4)-- (5,3);
\draw (5,3)-- (5,2);
\draw (6,2)-- (6,3);
\draw (6,3)-- (6,4);
\draw (4,4) .. controls (4.5,4.5) and (5.5,4.5) .. (6,4);
\draw (4,3) .. controls (4.5,3.5) and (5.5,3.5) .. (6,3);
\draw (4,2) .. controls (4.5,2.5) and (5.5,2.5) .. (6,2);
\draw (4,4) .. controls (4.5,3.5) and (4.5,2.5) .. (4,2);
\draw (5,4) .. controls (5.5,3.5) and (5.5,2.5) .. (5,2);
\draw (6,4) .. controls (6.5,3.5) and (6.5,2.5) .. (6,2);
\begin{scriptsize}
\fill [color=red] (4,4) circle (1.5pt);
\fill [color=red] (5,4) circle (1.5pt);
\fill [color=black] (6,4) circle (1.5pt);
\fill [color=black] (4,3) circle (1.5pt);
\fill [color=black] (5,3) circle (1.5pt);
\fill [color=red] (6,3) circle (1.5pt);
\fill [color=black] (4,2) circle (1.5pt);
\fill [color=black] (5,2) circle (1.5pt);
\fill [color=black] (6,2) circle (1.5pt);
\end{scriptsize}
\end{tikzpicture}\label{grith3x3c}}
\subfigure[Case $k=2$ (d)]{\begin{tikzpicture}[line cap=round,line join=round,>=triangle 45,x=1.0cm,y=1.0cm]
\clip(3.5,1.5) rectangle (6.5,4.5);
\draw (4,4)-- (5,4);
\draw (5,4)-- (6,4);
\draw (4,3)-- (5,3);
\draw (5,3)-- (6,3);
\draw (4,2)-- (5,2);
\draw (5,2)-- (6,2);
\draw (4,2)-- (4,3);
\draw (4,3)-- (4,4);
\draw (5,4)-- (5,3);
\draw (5,3)-- (5,2);
\draw (6,2)-- (6,3);
\draw (6,3)-- (6,4);
\draw (4,4) .. controls (4.5,4.5) and (5.5,4.5) .. (6,4);
\draw (4,3) .. controls (4.5,3.5) and (5.5,3.5) .. (6,3);
\draw (4,2) .. controls (4.5,2.5) and (5.5,2.5) .. (6,2);
\draw (4,4) .. controls (4.5,3.5) and (4.5,2.5) .. (4,2);
\draw (5,4) .. controls (5.5,3.5) and (5.5,2.5) .. (5,2);
\draw (6,4) .. controls (6.5,3.5) and (6.5,2.5) .. (6,2);
\begin{scriptsize}
\fill [color=red] (4,4) circle (1.5pt);
\fill [color=black] (5,4) circle (1.5pt);
\fill [color=black] (6,4) circle (1.5pt);
\fill [color=black] (4,3) circle (1.5pt);
\fill [color=red] (5,3) circle (1.5pt);
\fill [color=black] (6,3) circle (1.5pt);
\fill [color=black] (4,2) circle (1.5pt);
\fill [color=black] (5,2) circle (1.5pt);
\fill [color=red] (6,2) circle (1.5pt);
\end{scriptsize}
\end{tikzpicture}\label{grith3x3d}}
\subfigure[Cube]{\begin{tikzpicture}[line cap=round,line join=round,>=triangle 45,x=0.3cm,y=0.3cm]
\clip(4.5,0) rectangle (11.5,8.5);
\draw (5,2)-- (9,2);
\draw (5,2)-- (5,6);
\draw (5,6)-- (9,6);
\draw (9,6)-- (9,2);
\draw (5,6)-- (7,8);
\draw (7,8)-- (11,8);
\draw (11,8)-- (9,6);
\draw (11,8)-- (11,4);
\draw (9,2)-- (11,4);
\draw (5,2)-- (7,4);
\draw (7,4)-- (8.85,4);
\draw (9.15,4)-- (11,4);
\draw (7,8)-- (7,6.15);
\draw (7,4)-- (7,5.85);
\begin{scriptsize}
\fill [color=black] (5,2) circle (1.5pt);
\fill [color=black] (9,2) circle (1.5pt);
\fill [color=black] (5,6) circle (1.5pt);
\fill [color=black] (9,6) circle (1.5pt);
\fill [color=black] (7,8) circle (1.5pt);
\fill [color=black] (11,8) circle (1.5pt);
\fill [color=black] (11,4) circle (1.5pt);
\fill [color=black] (7,4) circle (1.5pt);
\end{scriptsize}
\end{tikzpicture}\label{cube}}
\subfigure[Cube (a)]{\begin{tikzpicture}[line cap=round,line join=round,>=triangle 45,x=0.3cm,y=0.3cm]
\clip(4.5,0) rectangle (11.5,8.5);
\draw (5,2)-- (9,2);
\draw (5,2)-- (5,6);
\draw (5,6)-- (9,6);
\draw (9,6)-- (9,2);
\draw (5,6)-- (7,8);
\draw (7,8)-- (11,8);
\draw (11,8)-- (9,6);
\draw (11,8)-- (11,4);
\draw (9,2)-- (11,4);
\draw (5,2)-- (7,4);
\draw (7,4)-- (8.85,4);
\draw (9.15,4)-- (11,4);
\draw (7,8)-- (7,6.15);
\draw (7,4)-- (7,5.85);
\begin{scriptsize}
\fill [color=black] (5,2) circle (1.5pt);
\fill [color=black] (9,2) circle (1.5pt);
\fill [color=red] (5,6) circle (1.5pt);
\fill [color=black] (9,6) circle (1.5pt);
\fill [color=red] (7,8) circle (1.5pt);
\fill [color=red] (11,8) circle (1.5pt);
\fill [color=black] (11,4) circle (1.5pt);
\fill [color=black] (7,4) circle (1.5pt);
\end{scriptsize}
\end{tikzpicture}\label{cubea}}
\subfigure[Cube (b)]{\begin{tikzpicture}[line cap=round,line join=round,>=triangle 45,x=0.3cm,y=0.3cm]
\clip(4.5,0) rectangle (11.5,8.5);
\draw (5,2)-- (9,2);
\draw (5,2)-- (5,6);
\draw (5,6)-- (9,6);
\draw (9,6)-- (9,2);
\draw (5,6)-- (7,8);
\draw (7,8)-- (11,8);
\draw (11,8)-- (9,6);
\draw (11,8)-- (11,4);
\draw (9,2)-- (11,4);
\draw (5,2)-- (7,4);
\draw (7,4)-- (8.85,4);
\draw (9.15,4)-- (11,4);
\draw (7,8)-- (7,6.15);
\draw (7,4)-- (7,5.85);
\begin{scriptsize}
\fill [color=red] (5,2) circle (1.5pt);
\fill [color=black] (9,2) circle (1.5pt);
\fill [color=black] (5,6) circle (1.5pt);
\fill [color=black] (9,6) circle (1.5pt);
\fill [color=red] (7,8) circle (1.5pt);
\fill [color=red] (11,8) circle (1.5pt);
\fill [color=black] (11,4) circle (1.5pt);
\fill [color=black] (7,4) circle (1.5pt);
\end{scriptsize}
\end{tikzpicture}\label{cubeb}}
\subfigure[Cube (c)]{\begin{tikzpicture}[line cap=round,line join=round,>=triangle 45,x=0.3cm,y=0.3cm]
\clip(4.5,0) rectangle (11.5,8.5);
\draw (5,2)-- (9,2);
\draw (5,2)-- (5,6);
\draw (5,6)-- (9,6);
\draw (9,6)-- (9,2);
\draw (5,6)-- (7,8);
\draw (7,8)-- (11,8);
\draw (11,8)-- (9,6);
\draw (11,8)-- (11,4);
\draw (9,2)-- (11,4);
\draw (5,2)-- (7,4);
\draw (7,4)-- (8.85,4);
\draw (9.15,4)-- (11,4);
\draw (7,8)-- (7,6.15);
\draw (7,4)-- (7,5.85);
\begin{scriptsize}
\fill [color=black] (5,2) circle (1.5pt);
\fill [color=red] (9,2) circle (1.5pt);
\fill [color=red] (5,6) circle (1.5pt);
\fill [color=black] (9,6) circle (1.5pt);
\fill [color=black] (7,8) circle (1.5pt);
\fill [color=red] (11,8) circle (1.5pt);
\fill [color=black] (11,4) circle (1.5pt);
\fill [color=black] (7,4) circle (1.5pt);
\end{scriptsize}
\end{tikzpicture}\label{cubec}}
\caption{Auxiliary graphs for Theorem \ref{theocmpltprd}}
\end{figure}

\begin{theorem}\label{theocmpltprd}
For $k\geq2$, if $n_i\geq2$ for all $1\leq i\leq k$, 
then $$BI_2(K(n_1,n_2,\dots,n_k))\simeq\bigvee_{F_k(n_1,\dots,n_k)}\mathbb{S}^{1}$$
$$\Delta_2^t(K(n_1,n_2,\dots,n_k))\simeq\bigvee_{F_k(n_1,\dots,n_k)}\mathbb{S}^{n-4}$$ 
where $n=n_1\cdots n_k$.
\end{theorem}
\begin{proof}
We start with the clique complex.
For the case $k=2$, the clique complex of $K(n_1,n_2)$ is the independence complex of categorical product $K_{n_1}\times K_{n_2}$ and 
the result for this case is Proposition 3.4 of \citep{homotopygoyal}.
Suppose the result is true for any $k-1$. For $k$ we use induction on $n_1+\cdots+n_k$. If $n_1=n_2=\cdots=n_k=1$, the result is clear. 
Notice that $K(n_1,n_2,\dots,n_{k-1},1)\cong K(n_1,n_2,\dots,n_{k-1})$, thus we can assume $n_i\geq2$ for all $i$.
If $n_1=n_2=\cdots=n_k=2$, then $BI_2(K(2,2,\dots,2))$ is the $1$-skeleton of the $k$ hypercube and has the homotopy type of 
$k2^{k-1}-2^k+1$ circles. Suppose the result is true for $n_1+\cdots+n_k-1$. 
The clique complex $BI_2(K(n_1,n_2,\dots,n_k))$ have as facets the simplices
$$\sigma_j(i_1,\dots,i_k)=\prod_{r=1}^kX_{i_r}$$ 
where $i_j=*$, $i_r\in\underline{n_{r}}$ for $r\neq j$, $X_j=\underline{n_j}$ and $X_r=\{i_r\}$. We take the subcomplex 
$J_1$ with facets all the facets $\sigma_j(i_1,\dots,i_k)$ such that one of the following conditions is achieve:
\begin{itemize}
    \item[(a)] $j=k$ and no restriction on the other parameters.
    \item[(b)] $j<k$ with $i_k<n_k$ and no restriction on the other parameters.
\end{itemize}
Then $J_1\simeq BI_2(K(n_1,n_2,\dots,n_k-1))$. We also take the subcomplex $J_2$ with facets all the facets $\sigma_j(i_1,\dots,i_k)$ such that 
one of the following conditions is achieved:
\begin{itemize}
    \item[(a)] $j=k$ and no restriction on the other parameters.
    \item[(c)] $j<k$, $i_k=n_k$ and no restriction on the other parameters.
\end{itemize}
Then $J_2\simeq BI_2(K(n_1,n_2,\dots,1))$. Notice that the intersection $J_1\cap J_2$ is the complex with facets of type (a), thus is homotopy equivalent to the disjoin union of 
$n_1\cdots n_{k-1}$ points. Now, we have that
$$BI_2(K(n_1,n_2,\dots,n_k))=J_1\cup J_2\;\;\mbox{ and }\;\;J_1\cap J_2\simeq\bigvee_{n_1\cdots n_{k-1}-1}\mathbb{S}^0$$
Therefore, by Lemma \ref{homocolimpegado}, we have that
$$BI_2(K(n_1,n_2,\dots,n_k))\simeq BI_2(K(n_1,n_2,\dots,n_k-1))\vee BI_2(K(n_1,n_2,\dots,n_{k-1},1))\vee\bigvee_{n_1\cdots n_{k-1}-1}\mathbb{S}^1$$
By inductive hypothesis the first two spaces of the wedge have the homotopy type expected and using the properties of the functions $F_k$'s we obtain the result.

Now we calculate the homotopy type for the total cut complex. 
For $k=2$, $n_1\geq2$ and $n_2=2$ the result is true \citep[see Theorem 4.8]{totcutcompl}. For the remaining cases we will show that $\Delta_2^t(K(n_1,n_2,\dots,n_k))$ is simply connected. 
For $k=2$ we can assume that $n_1,n_2\geq3$ and we take $x,y,z$ distinct vertices. There 
is a subgraph isomorphic to $K(3,3)$ (see Figure \ref{grith3x3}) that contains $\{x,y,z\}$. Up to isomorphism there are $4$ possible cases shown in Figures 
\ref{grith3x3a},\ref{grith3x3b},\ref{grith3x3c},\ref{grith3x3d} where the vertices of $\{x,y,z\}$ are the red ones. Regardless of the case, we have that there is a independent set of cardinality 
$2$ without the vertices $x,y,z$. Therefore $\mathrm{sk}_2\Delta_2^t(K(n_1,n_2))=\mathrm{sk}_2\Delta^{\underline{n_1}\times\underline{n_2}}$. For $k\geq3$, we take the subgraph $Q=K(2,2,2,1,\dots,1)$ which 
is isomorphic to a cube (see Figure \ref{cube}). We take $x,y,z$ distinct vertices. If $|\{x,y,z\}\cap V(Q)|\leq2$, it is clear that $\{x,y,z\}$ is a simplex. Assume 
$|\{x,y,z\}\cap V(Q)|=3$, then there are up to isomorphism $3$ cases shown in Figures \ref{cubea},\ref{cubeb},\ref{cubec} where the vertices $x,y,z$ are the red ones. Regardless of the case, 
we have that there is a independent set of cardinality $2$ without the vertices $x,y,z$. Therefore $\mathrm{sk}_2\Delta_2^t(K(n_1,n_2,\dots,n_k))=\mathrm{sk}_2\Delta^{\underline{n_1}\times\dots\times\underline{n_k}}$.
By Theorems \ref{dualidadalexander} and \ref{homttyphomlgrp}, $\Delta_2^t(K(n_1,n_2,\dots,n_k))$ has the expected homotopy type.
\end{proof}

We finish with the following question.

\begin{que}
Which are the homotopy types of $BI_d(K(n_1,\dots,n_l))$ and $\Delta_d^t(K(n_1,\dots,n_l))$ for $d\geq3$?
\end{que}

\bibliographystyle{acm}
\bibliography{biblio}

@InBook{bjornertopmeth,
  chapter   = {Topological methods},
  pages     = {1819-1872},
  title     = {Handbook of Combinatorics},
  publisher = {North-Holland},
  year      = {1995},
  author    = {Anders Björner},
  editor    = {R. Graham and M. Grötschel and L. Lov\'asz},
}

@Book{cubicalhomotopy,
  AUTHOR = {Munson, Brian A. and Voli\'{c}, Ismar},
     TITLE = {Cubical homotopy theory},
    SERIES = {New Mathematical Monographs},
    VOLUME = {25},
 PUBLISHER = {Cambridge University Press, Cambridge},
      YEAR = {2015},
     PAGES = {xv+631},
      ISBN = {978-1-107-03025-1},
   MRCLASS = {55P99 (16B50 18G30 55U10 57Q45 57R19)},
  MRNUMBER = {3559153},
MRREVIEWER = {Rui\ Miguel\ Saramago},
       DOI = {10.1017/CBO9781139343329},
       URL = {https://doi.org/10.1017/CBO9781139343329},
}

@Book{graphsanddigraphs,
  AUTHOR = {Chartrand, Gary and Lesniak, Linda and Zhang, Ping},
     TITLE = {Graphs \& digraphs},
    SERIES = {Textbooks in Mathematics},
   EDITION = {Sixth},
 PUBLISHER = {CRC Press, Boca Raton, FL},
      YEAR = {2016},
     PAGES = {xii+628},
      ISBN = {978-1-4987-3576-6},
   MRCLASS = {05-01 (05Cxx 05D10 05D40)},
  MRNUMBER = {3445306},
}

@article {clicomgrppowers,
    AUTHOR = {Adamaszek, Micha\l },
     TITLE = {Clique complexes and graph powers},
   JOURNAL = {Israel J. Math.},
  FJOURNAL = {Israel Journal of Mathematics},
    VOLUME = {196},
      YEAR = {2013},
    NUMBER = {1},
     PAGES = {295--319},
      ISSN = {0021-2172,1565-8511},
   MRCLASS = {05C69 (05C76 05E45 55U10)},
  MRNUMBER = {3096593},
MRREVIEWER = {Andrej\ Taranenko},
       DOI = {10.1007/s11856-012-0166-1},
       URL = {https://doi.org/10.1007/s11856-012-0166-1},
}

@article {kimboundind,
    AUTHOR = {Kim, Minki and Lew, Alan},
     TITLE = {Complexes of graphs with bounded independence number},
   JOURNAL = {Israel J. Math.},
  FJOURNAL = {Israel Journal of Mathematics},
    VOLUME = {249},
      YEAR = {2022},
    NUMBER = {1},
     PAGES = {83--120},
      ISSN = {0021-2172,1565-8511},
   MRCLASS = {05C69},
  MRNUMBER = {4462630},
MRREVIEWER = {G\'{a}bor\ N.\ S\'{a}rk\"{o}zy},
       DOI = {10.1007/s11856-022-2308-4},
       URL = {https://doi.org/10.1007/s11856-022-2308-4},
}

@article {totcutcompl,
    AUTHOR = {Bayer, Margaret and Denker, Mark and Milutinovi\'c, Marija
              Jeli\'c{} and Rowlands, Rowan and Sundaram, Sheila and Xue,
              Lei},
     TITLE = {Total {C}ut {C}omplexes of {G}raphs},
   JOURNAL = {Discrete Comput. Geom.},
  FJOURNAL = {Discrete \& Computational Geometry. An International Journal
              of Mathematics and Computer Science},
    VOLUME = {73},
      YEAR = {2025},
    NUMBER = {2},
     PAGES = {500--527},
      ISSN = {0179-5376,1432-0444},
   MRCLASS = {57M15 (05C69 05E45 57Q70)},
  MRNUMBER = {4865931},
       DOI = {10.1007/s00454-024-00630-4},
       URL = {https://doi.org/10.1007/s00454-024-00630-4},
}

@Book{hatcher,
  AUTHOR = {Hatcher, Allen},
     TITLE = {Algebraic topology},
 PUBLISHER = {Cambridge University Press, Cambridge},
      YEAR = {2002},
     PAGES = {xii+544},
      ISBN = {0-521-79160-X; 0-521-79540-0},
   MRCLASS = {55-01 (55-00)},
  MRNUMBER = {1867354},
MRREVIEWER = {Donald\ W.\ Kahn},
}

@Article{bjorneralexander,
 AUTHOR = {Bj\"orner, Anders and Tancer, Martin},
     TITLE = {Note: {C}ombinatorial {A}lexander duality---a short and
              elementary proof},
   JOURNAL = {Discrete Comput. Geom.},
  FJOURNAL = {Discrete \& Computational Geometry. An International Journal
              of Mathematics and Computer Science},
    VOLUME = {42},
      YEAR = {2009},
    NUMBER = {4},
     PAGES = {586--593},
      ISSN = {0179-5376,1432-0444},
   MRCLASS = {55U30 (05E45)},
  MRNUMBER = {2556456},
       DOI = {10.1007/s00454-008-9102-x},
       URL = {https://doi.org/10.1007/s00454-008-9102-x},
}

@misc{froberg,
 author = {Fr{\"o}berg, Ralf},
 title = {On {Stanley}-{Reisner} rings},
 year = {1990},
 language = {English},
 howpublished = {Topics in algebra. {Part} 2: {Commutative} rings and algebraic groups, {Pap}. 31st {Semester} {Class}. {Algebraic} {Struct}., {Warsaw}/{Poland} 1988, {Banach} {Cent}. {Publ}. 26, {Part} 2, 57-70 (1990).},
 zbMATH = {16165},
 Zbl = {0741.13006}
}

@article{cutcomplexesi,
 author = {Bayer, Margaret and Denker, Mark and Milutinovi{\'c}, Marija Jeli{\'c} and Rowlands, Rowan and Sundaram, Sheila and Xue, Lei},
 title = {Topology of cut complexes of graphs},
 fjournal = {SIAM Journal on Discrete Mathematics},
 journal = {SIAM J. Discrete Math.},
 issn = {0895-4801},
 volume = {38},
 number = {2},
 pages = {1630--1675},
 year = {2024},
 language = {English},
 doi = {10.1137/23M1569034},
 zbMATH = {7862216},
 Zbl = {1550.57012}
}

@article {homtphmlgrps,
    AUTHOR = {Carnero$\;$Bravo, Andr\'es and Antol\'in$\;$Camarena, Omar},
     TITLE = {Homotopy type through homology groups},
   JOURNAL = {Bol. Soc. Mat. Mex. (3)},
  FJOURNAL = {Bolet\'in de la Sociedad Matem\'atica Mexicana. Third Series},
    VOLUME = {30},
      YEAR = {2024},
    NUMBER = {2},
     PAGES = {Paper No. 28, 6},
      ISSN = {1405-213X,2296-4495},
   MRCLASS = {55P15 (55N10 55P10)},
  MRNUMBER = {4715353},
       DOI = {10.1007/s40590-024-00605-8},
       URL = {https://doi.org/10.1007/s40590-024-00605-8},
}

@book{combcompwords,
 author = {Heubach, Silvia and Mansour, Toufik},
 title = {Combinatorics of compositions and words.},
 fseries = {Discrete Mathematics and its Applications},
 series = {Discrete Math. Appl. (Boca Raton)},
 isbn = {978-1-4200-7267-9; 978-1-138-11667-2; 978-1-4200-7268-6},
 year = {2009},
 publisher = {Boca Raton, FL: CRC Press},
 language = {English},
 doi = {10.1201/9781420072686},
 zbMATH = {5590728},
 Zbl = {1184.68373}
}

@article{repfntgrpstintRl,
 author = {Lekkerkerker, C. G. and Boland, J. Ch.},
 title = {Representation of a finite graph by a set of intervals on the real line},
 fjournal = {Fundamenta Mathematicae},
 journal = {Fundam. Math.},
 issn = {0016-2736},
 volume = {51},
 pages = {45--64},
 year = {1962},
 language = {English},
 doi = {10.4064/fm-51-1-45-64},
 url = {https://eudml.org/doc/213681},
 zbMATH = {3171679},
 Zbl = {0105.17501}
}

@Article{chauhantot2comcycl,
  author = {Chauhan, Pratiksha and Shukla, Samir and Vinayak, Kumar},
  title = {Total 2-cut complexes of powers of cycle graphs and Cartesian products of certain graphs},
  journal     ={https://arxiv.org/abs/2512.04486},
  year        = {2025},
  eprint      = {2512.04486},
  eprintclass = {math.AT},
  eprinttype  = {arXiv},
  url        = {https://arxiv.org/abs/2512.04486},
}

@Article{chandrakartoptotcutgrid,
 author = {Chandrakar, Himanshu and Hazra, Nisith Ranjan and Rout, Debotosh and Singh, Anurag},
 title = {Topology of total cut complexes and cut complexes of grid graphs},
 fjournal = {SIAM Journal on Discrete Mathematics},
 journal = {SIAM J. Discrete Math.},
 issn = {0895-4801},
 volume = {40},
 number = {2},
 pages = {760--788},
 year = {2026},
 language = {English},
 doi = {10.1137/24M1687716},
 zbMATH = {8207264}
}

@Article{shenhmtpcyclepwr,
  author = {Shen, Yufeng and Song, Zhiyu and Yu, Fenglin and Zhou, Leopold Wuhan and Zhuang, Jingqi},
  title = {Homotopy {Type} of {Total} {Cut} {Complexes} of {Squared} {Cycle} {Graphs}},
  journal     ={https://arxiv.org/abs/2510.22574},
  year        = {2025},
  eprint      = {2510.22574},
  eprintclass = {math.AT},
  eprinttype  = {arXiv},
  url        = {https://arxiv.org/abs/2510.22574},
}

@Article{filak,
  author = {Filakovsk{\'y}, Marek},
  title = {The {Topology} of $k$-{Robust} {Clique} {Complexes} in {Grid}-like {Graphs}},
  year = {2026},
  journal     ={https://arxiv.org/abs/2602.11365},
  eprint      = {2602.11365},
  eprintclass = {math.AT},
  eprinttype  = {arXiv},
  url        = {https://arxiv.org/abs/2602.11365},
}

@Article{shella,
  author = {Ghosh, Rakesh and Selvaraja, S},
  title = {Shellability in {Clique}-{Free} {Complexes} of {Graphs}},
  year = {2026},
  journal     ={https://arxiv.org/abs/2602.09623},
  eprint      = {2602.09623},
  eprintclass = {math.AT},
  eprinttype  = {arXiv},
  url        = {https://arxiv.org/abs/2602.09623},
}

@article{zieglerhocolim,
AUTHOR = {Welker, Volkmar and Ziegler, G\"{u}nter M. and
              \v{Z}ivaljevi\'{c}, Rade T.},
     TITLE = {Homotopy colimits---comparison lemmas for combinatorial
              applications},
   JOURNAL = {J. Reine Angew. Math.},
  FJOURNAL = {Journal f\"{u}r die Reine und Angewandte Mathematik. [Crelle's
              Journal]},
    VOLUME = {509},
      YEAR = {1999},
     PAGES = {117--149},
      ISSN = {0075-4102,1435-5345},
   MRCLASS = {55P99 (05B30)},
  MRNUMBER = {1679169},
MRREVIEWER = {R.\ M.\ Vogt},
       DOI = {10.1515/crll.1999.035},
       URL = {https://doi.org/10.1515/crll.1999.035},
}

@article{dugger2008primer,
  title={A primer on homotopy colimits},
  author={Dugger, Daniel},
  journal={preprint},
  year={2008}
}

@Article{homotopygoyal,
  AUTHOR = {Goyal, Shuchita and Shukla, Samir and Singh, Anurag},
     TITLE = {Homotopy type of independence complexes of certain families of
              graphs},
   JOURNAL = {Contrib. Discrete Math.},
  FJOURNAL = {Contributions to Discrete Mathematics},
    VOLUME = {16},
      YEAR = {2021},
    NUMBER = {3},
     PAGES = {74--92},
      ISSN = {1715-0868},
   MRCLASS = {05C69 (55P15)},
  MRNUMBER = {4369845},
MRREVIEWER = {Sel\c{c}uk\ Kayacan},
DOI = {10.11575/cdm.v16i3.71284},
}

@article{mcordsnghmgrpshmty,
 author = {McCord, M. C.},
 title = {Singular homology groups and homotopy groups of finite topological spaces},
 fjournal = {Duke Mathematical Journal},
 journal = {Duke Math. J.},
 issn = {0012-7094},
 volume = {33},
 pages = {465--474},
 year = {1966},
 language = {English},
 doi = {10.1215/S0012-7094-66-03352-7},
 zbMATH = {3229383},
 Zbl = {0142.21503}
}

@article{dochtermanengstrom,
 author = {Dochtermann, Anton and Engstr{\"o}m, Alexander},
 title = {Algebraic properties of edge ideals via combinatorial topology},
 fjournal = {The Electronic Journal of Combinatorics},
 journal = {Electron. J. Comb.},
 issn = {1077-8926},
 volume = {16},
 number = {2},
 pages = {research paper r2, 24},
 year = {2009},
 language = {English},
 url = {https://eudml.org/doc/129889},
 zbMATH = {5541104},
 Zbl = {1161.13013}
}

@article{prisner,
 author = {Prisner, Erich},
 title = {Convergence of iterated clique graphs},
 fjournal = {Discrete Mathematics},
 journal = {Discrete Math.},
 issn = {0012-365X},
 volume = {103},
 number = {2},
 pages = {199--207},
 year = {1992},
 language = {English},
 doi = {10.1016/0012-365X(92)90270-P},
 zbMATH = {89926},
 Zbl = {0766.05096}
}

@article{friedmansimsets,
 author = {Friedman, Greg},
 title = {Survey article: an elementary illustrated introduction to simplicial sets},
 fjournal = {Rocky Mountain Journal of Mathematics},
 journal = {Rocky Mt. J. Math.},
 issn = {0035-7596},
 volume = {42},
 number = {2},
 pages = {353--423},
 year = {2012},
 language = {English},
 doi = {10.1216/RMJ-2012-42-2-353},
 zbMATH = {6035442},
 Zbl = {1248.55001}
}

@article{minianalxdual,
 author = {Minian, Elías Gabriel and Rodríguez, Jorge Tomás},
 title = {A note on the homotopy type of the {Alexander} dual},
 fjournal = {Discrete \& Computational Geometry},
 journal = {Discrete Comput. Geom.},
 issn = {0179-5376},
 volume = {52},
 number = {1},
 pages = {34--43},
 year = {2014},
 language = {English},
 doi = {10.1007/s00454-014-9606-5},
 zbMATH = {6346216},
 Zbl = {1301.55013}
}

@article{zaare2026vertex,
  title={Vertex Decomposability of Some Complexes Associated to Graphs},
  author={Zaare-Nahandi, Rahim},
  journal={Bulletin of the Iranian Mathematical Society},
  volume={52},
  number={4},
  pages={40},
  year={2026},
  publisher={Springer}
}

@article{froberg24,
 author = {Fr{\"o}berg, Ralf},
 title = {On {Stanley}-{Reisner} rings with linear resolutions},
 fjournal = {Bulletin of the Iranian Mathematical Society},
 journal = {Bull. Iran. Math. Soc.},
 issn = {1017-060X},
 volume = {50},
 number = {6},
 pages = {13},
 note = {Id/No 86},
 year = {2024},
 language = {English},
 doi = {10.1007/s41980-024-00947-z},
 zbMATH = {8027112},
 Zbl = {1573.13045}
}

@article{eagonreiner,
 author = {Eagon, John A. and Reiner, Victor},
 title = {Resolutions of {Stanley}-{Reisner} rings and {Alexander} duality},
 fjournal = {Journal of Pure and Applied Algebra},
 journal = {J. Pure Appl. Algebra},
 issn = {0022-4049},
 volume = {130},
 number = {3},
 pages = {265--275},
 year = {1998},
 language = {English},
 doi = {10.1016/S0022-4049(97)00097-2},
 zbMATH = {1308214},
 Zbl = {0941.13016}
}

@article{lovasz,
 author = {Lov{\'a}sz, L{\'a}szl{\'o}},
 title = {Kneser's conjecture, chromatic number, and homotopy},
 fjournal = {Journal of Combinatorial Theory. Series A},
 journal = {J. Comb. Theory, Ser. A},
 issn = {0097-3165},
 volume = {25},
 pages = {319--324},
 year = {1978},
 language = {English},
 doi = {10.1016/0097-3165(78)90022-5},
 zbMATH = {3650589},
 Zbl = {0418.05028}
}

@incollection{haxelltopcinind,
 author = {Haxell, Penny},
 title = {Topological connectedness and independent sets in graphs},
 booktitle = {Surveys in combinatorics 2019. Papers based on the 27th British combinatorial conference, University of Birmingham, Birmingham, UK, July 29 -- August 2, 2019},
 isbn = {978-1-108-74072-2; 978-1-108-64909-4},
 pages = {89--113},
 year = {2019},
 publisher = {Cambridge: Cambridge University Press},
 language = {English},
 doi = {10.1017/9781108649094.004},
 zbMATH = {7307103},
 Zbl = {1476.05156}
}

@book{kozlovcombalgtop,
 author = {Kozlov, Dimitry},
 title = {Combinatorial algebraic topology},
 fseries = {Algorithms and Computation in Mathematics},
 series = {Algorithms Comput. Math.},
 issn = {1431-1550},
 volume = {21},
 isbn = {978-3-540-71961-8; 978-3-540-73051-4},
 year = {2008},
 publisher = {Berlin: Springer},
 language = {English},
 zbMATH = {5175082},
 Zbl = {1130.55001}
}

@book{jonsson,
 author = {Jonsson, Jakob},
 title = {Simplicial complexes of graphs},
 fseries = {Lecture Notes in Mathematics},
 series = {Lect. Notes Math.},
 issn = {0075-8434},
 volume = {1928},
 isbn = {978-3-540-75858-7},
 year = {2008},
 publisher = {Berlin: Springer},
 language = {English},
 doi = {10.1007/978-3-540-75859-4},
 zbMATH = {5209738},
 Zbl = {1152.05001}
}

@book{delongueville,
 author = {de Longueville, Mark},
 title = {A course in topological combinatorics},
 fseries = {Universitext},
 series = {Universitext},
 issn = {0172-5939},
 isbn = {978-1-4419-7909-4; 978-1-4419-7910-0},
 year = {2013},
 publisher = {New York, NY: Springer},
 language = {English},
 doi = {10.1007/978-1-4419-7910-0},
 zbMATH = {5877216},
 Zbl = {1273.05001}
}
\end{document}